\documentclass[11pt]{amsart}
\usepackage{amssymb,amsmath,amsthm}
\usepackage{amsfonts}
\usepackage{enumerate}
\usepackage{dsfont}
\usepackage{cite}
\usepackage[colorlinks, citecolor=red]{hyperref}
\usepackage{mathrsfs}
\usepackage{epsfig}
\usepackage{lscape}
\usepackage{subfigure}
\usepackage{epstopdf}
\usepackage{caption}
\usepackage{algorithm}
\usepackage{algpseudocode}
\usepackage{multirow}
\usepackage{geometry}
\usepackage{paralist}
\usepackage{enumerate}
\usepackage{url}
\usepackage{graphicx,cite}
\usepackage{longtable}
\usepackage{tikz}

\newcommand{\rank}{{\rm rank}}

\newcommand{\diag}{{\rm diag}}

\newtheorem{definition}{Definition}[section]

\newtheorem{prop}[definition]{Proposition}
\newtheorem{theorem}[definition]{Theorem}
\newtheorem{lemma}[definition]{Lemma}

\newtheorem{remark}[definition]{Remark}

\date{}

\begin{document}
\baselineskip 18pt
\bibliographystyle{plain}
\title[Efficient variance reduction for RK]{Efficient variance reduction for randomized Kaczmarz methods via residual-tracking control variates}

\author{Jiaxin Xie}
\address{LMIB of the Ministry of Education, School of Mathematical Sciences, Beihang University, Beijing, 100191, China. }
\email{xiejx@buaa.edu.cn}

\author{Hou-Duo Qi}
\address{Department of Data Science and Artificial Intelligence and Department of Applied Mathematics, The Hong Kong Polytechnic University, Hung Hom, Kowloon, Hong Kong. }
\email{houduo.qi@polyu.edu.hk}

\author{Deren Han}
\address{LMIB of the Ministry of Education, School of Mathematical Sciences, Beihang University, Beijing, 100191, China. }
\email{handr@buaa.edu.cn}

\begin{abstract}
 The randomized Kaczmarz (RK) method is a widely used and effective class of iterative algorithms for solving large-scale linear systems.
 For inconsistent systems, however, the persistent variance caused by randomization prevents the  method from converging exactly to the least-squares solution. In this paper, we introduce a residual-tracking control variate framework for reducing this variance while retaining the row-action structure and low per-iteration cost of RK. 
 The resulting variance-reduced RK method uses constant step sizes, accesses only sampled rows of the coefficient matrix, and requires \(O(m+n)\) auxiliary storage.
 We prove that the proposed method converges linearly in expectation to the unique least-norm solution for general linear systems, including inconsistent and rank-deficient cases, with a convergence factor of the same order as that of the RK method. Furthermore, we incorporate the momentum acceleration into our framework and provide theoretical guarantees for its improved convergence rate. Numerical experiments  illustrate the effectiveness of the proposed methods, particularly for ill-conditioned systems and high-accuracy computations.
\end{abstract}

\maketitle

\let\thefootnote\relax

\let\thefootnote\relax

\section{Introduction}

We consider the problem of solving the linear system
\begin{equation}
	\label{main-prob}
	Ax = b, \ \ A \in \mathbb{R}^{m \times n}, \ b \in \mathbb{R}^m,  
\end{equation}
 under various problem settings. It is well known  \cite{golub2013matrix,ben2003generalized,Du20Ran} that this system admits a unique minimum-norm solution in the least-squares sense, given by \( A^\dagger b \), where \( A^\dagger \) denotes the Moore-Penrose pseudoinverse. This solution exists regardless of the dimensions, rank, or consistency of the system. However, computing this solution directly can be computationally expensive and memory-intensive, especially when the number of rows is large. In extremely large-scale problems (e.g., \( m \geq 10^{12} \)), even storing a single column of \( A \) in random-access memory (RAM) may become prohibitive.

\emph{Row-action} methods have been proposed as practical approaches for solving large-scale linear systems. These methods access and process one or a few rows of \( A \) at a time. A prominent example of a row-action method is the randomized Kaczmarz (RK) method \cite{Str09}.  
When the linear system \eqref{main-prob} is consistent, the RK method converges linearly in expectation to the minimum-norm least-squares solution \( A^\dagger b \), provided an appropriate choice of initial point \cite{Str09,Han2022-xh,DS2021}. However, if the system is inconsistent, Needell \cite{needell2010randomized} showed that the RK method converges only to within a fixed radius (convergence horizon) of the least-squares solution.

It is worth noting that the RK method can be interpreted as a special case of stochastic gradient descent (SGD) \cite{robbins1951stochastic,Han2022-xh,needell2014stochasticMP,ma2017stochastic,zeng2023randomized,xie2025randomized}, where each update uses the gradient of a randomly selected row. Therefore, like SGD with a constant stepsize \cite{nemirovski2009robust,bottou2018optimization}, the persistent variance in the stochastic updates prevents RK from converging exactly to the least-squares solution in the inconsistent case, resulting in the observed convergence horizon.
A brief explanation of this phenomenon can be found in Section \ref{section-in-variance}.




\subsection{Our contribution}

In this paper, we propose an efficient variance reduction strategy tailored for the RK method, to compute the unique minimum-norm least-squares solution of linear systems.  The main contributions of this work are as follows.
 
 \begin{itemize}
 	\item[1.]  
 	We propose a novel residual-tracking control variate framework for the design and analysis of variance-reduced RK (VRRK) methods. The key idea is to construct an iteratively updated sequence that serves as a control variate for approximating the unknown optimal residual. Our framework requires access only to the rows of the coefficient matrix $A$ and permits the use of constant stepsizes, thereby overcoming several limitations of existing approaches. In particular, the randomized extended Kaczmarz (REK) method \cite{Du20Ran,ma2015convergence,Zou12} requires access to both the rows and columns of $A$, whereas other variance-reduced variants rely on diminishing stepsizes \cite{robbins1951stochastic,garrigos2023handbook}, which typically result in slow convergence. Moreover, the proposed method requires only $O(m+n)$ memory.
 	\item[2.] We prove that the proposed VRRK method converges linearly in expectation, with a convergence factor that explicitly captures the effects of the sampling probabilities, row norms, and nonzero singular values. We further establish that the VRRK-AS iteration can be interpreted as both a Jacobian sketching method \cite{gower2021stochastic} and a SAGA-type method \cite{defazio2014saga} applied to least-squares problems. Nevertheless, our convergence analysis is carried out directly in the original variables and establishes linear convergence to the minimum-norm least-squares solution, even when the coefficient matrix is rank deficient. 
 	\item[3.] We further introduce acceleration techniques to speed up the proposed VRRK method, inspired by existing accelerated variance-reduced frameworks \cite{allen2018katyusha,kovalev2020don,mishchenko2025distributed,pmlr-v119-li20g}.  We prove that with appropriately chosen parameters, the resulting accelerated VRRK method achieves a quadratic improvement in the convergence factor. Numerical experiments are also provided to confirm our results.
 \end{itemize}

\subsection{Related work}

\subsubsection{REK-type methods}

To address the divergence of the standard RK method for inconsistent linear systems, Zouzias and Freris \cite{Zou12} introduced a modification to the RK method and proposed the REK method. Later, Du et al. \cite{Du19} further improved the original REK framework and derived a sharper convergence factor. 
Let $a_i$ and $A_j$ denote the $i$-th row and $j$-th column of the matrix $A$, respectively, let $\|\cdot\|_2$ denote the Euclidean norm, and let $b_i$ denote the $i$-th entry of the vector $b$.
The iteration scheme of the REK method is presented as follows:
\begin{equation} \label{REK}
	\begin{aligned}
		z^{k+1} &= z^k - \frac{\langle A_{j_k}, z^k \rangle}{\|A_{j_k}\|_2^2} A_{j_k}, \\
		x^{k+1} &= x^k - \frac{\langle a_{i_k}, x^k \rangle - b_{i_k} + z^{k+1}_{i_k}}{\|a_{i_k}\|_2^2} a_{i_k},
	\end{aligned}
\end{equation}
where $i_k$ and $j_k$ are selected with probabilities proportional to $\|a_{i_k}\|_2^2$ and $\|A_{j_k}\|_2^2$, respectively.

Subsequently, a large body of work has emerged on the development and extension of the REK method, including block or deterministic variants of REK \cite{wu2021semiconvergence,zuo2025randomized,needell2013two,Du20Ran,wu2022two,popa1998extensions,popa1999characterization}, the randomized extended Gauss-Seidel (REGS) method \cite{Du19,ma2015convergence}, and the heavy-ball momentum accelerated REK method \cite{zeng2025adaptive}, among others. Although these methods converge to the minimum-norm least-squares solution, they share the same limitation as the iteration scheme \eqref{REK}, where each step requires access to both rows and columns of $A$. This significantly restricts the applicability of REK-type methods, particularly in scenarios where only row information is available. The proposed VRRK method in this paper requires only row information and achieves linear convergence, thereby eliminating this limitation of REK-type methods.

Finally, we note that other randomized methods based on coordinate descent (CD) for solving inconsistent systems include the randomized CD (RCD) method \cite{Lev10} and the randomized CGLS (RCGLS) method \cite{zeng2026randomized}. However, these methods converge to the minimum-norm least-squares solution only when the coefficient matrix $A$ has full column rank.  The proposed VRRK method works for general linear systems, regardless of the rank of $A$ or the consistency of the system.

\subsubsection{Variance-reduced methods}

SGD \cite{robbins1951stochastic} computes only one gradient per iteration, which makes it the workhorse for large-scale tasks. However, the large variance in its update directions requires diminishing step sizes to achieve exact convergence without residual bias. This drawback, however, yields a slow sublinear convergence rate \cite{moulines2011non}. 
A breakthrough in variance reduction (VR) was proposed in \cite{roux2012stochastic}, where the stochastic average gradient (SAG) method is introduced. Similar to SGD, SAG evaluates only one component gradient at each iteration, yet it leverages stored historical gradients from all components to maintain a precise approximation of the full objective gradient.
Remarkably, this approach retains the linear convergence rate of gradient descent. However, the proof of this result in \cite{schmidt2017minimizing} is notoriously challenging and requires computer-aided analysis.

Subsequently, SAGA \cite{defazio2014saga} was developed as a variant of SAG that uses control variates to construct an unbiased gradient estimator, offering comparable performance while simplifying the convergence analysis. In their general implementations, both SAG and SAGA store auxiliary gradient vectors for individual component functions, which can incur substantial memory costs when the number of components or the problem dimension is large. SVRG \cite{johnson2013accelerating} addresses this storage requirement by using a reference point and its full gradient. However, it evaluates two component gradients per iteration and recomputes the full gradient whenever the reference point is updated.

Our VRRK method also uses control variates but, unlike SVRG, does not require periodic full gradient evaluations. Instead, it updates the control variate incrementally using only the sampled rows. For least-squares problems, a structure-exploiting implementation of SAGA can likewise store scalar coefficients rather than full gradient vectors. Our residual-tracking formulation exploits this same structure by maintaining an $m$-dimensional residual tracker and an $n$-dimensional aggregate, requiring only $O(m+n)$ auxiliary memory. The main contributions of the present work therefore lie in the row-action interpretation of this formulation, its convergence analysis for minimum-norm least-squares solutions, and its accelerated extension.

More broadly, the literature on stochastic variance reduction is extensive; we refer the reader to the survey \cite{gower2020variance} for a comprehensive overview. Beyond SAGA and SVRG, representative methods include SDCA \cite{shalev2013stochastic}, S2GD \cite{konevcny2017semi}, MISO \cite{mairal2015incremental}, Finito \cite{defazio2014finito}, SARAH \cite{nguyen2017sarah}, and JacSketch \cite{gower2021stochastic}, for which linear convergence guarantees have been established under suitable smoothness and strong convexity assumptions.

\subsubsection{Momentum acceleration}

Starting with Katyusha \cite{allen2018katyusha}, a series of accelerated stochastic gradient algorithms have been developed to further improve the efficiency of classic variance-reduced methods \cite{zhou2019direct,zhou2018simple,NEURIPS2018_79a49b3e,pmlr-v80-allen-zhu18a,allen2017natasha,lan2019unified,kovalev2020don,mishchenko2025distributed,pmlr-v119-li20g}. These acceleration frameworks incorporate momentum-based update rules and refined sampling strategies to achieve improved convergence performance. We refer the reader to \cite{driggs2022accelerating} for further discussion.

In particular, the Katyusha method was originally proposed in \cite{allen2018katyusha} as an accelerated variant of SVRG and later optimized as a loopless Katyusha variant in \cite{kovalev2020don}. More recently, Li et al. \cite{pmlr-v119-li20g} developed an accelerated DIANA framework. As a classical variance reduction technique for compressed gradient descent, DIANA is compatible with typical ternary compression operators and has been widely applied in distributed optimization tasks \cite{mishchenko2025distributed}. Inspired by these effective acceleration and variance reduction strategies, we incorporate analogous momentum-based acceleration designs into our VRRK algorithm to further improve the convergence performance.

\subsection{Organization} 
 The remainder of the paper is organized as follows. Section 2 introduces the notation, arbitrary sampling, and the variance-reduction perspective. Section 3 develops the residual-tracking framework and presents the VRRK-AS method. Section 4 introduces the accelerated AVRRK-AS method and establishes its convergence guarantee. Numerical experiments are reported in Section 5, followed by concluding remarks in Section 6. The proofs of the main results are collected in the appendix.

\section{Preliminaries}
\subsection{Notations}
Throughout the paper, for any random variables $\xi$ and $\zeta$, we use $\mathbb{E}[\xi]$ and $\mathbb{E}[\xi\lvert \zeta]$ to denote the expectation of $\xi$ and the conditional expectation of $\xi$ given $\zeta$. 
For any vector $x\in\mathbb{R}^n$, we use $x_i$, $x^\top$, and $\|x\|_2$ to denote the $i$-th entry, the transpose and the Euclidean norm of $x$, respectively.
For any matrix $A\in\mathbb{R}^{m\times n}$, we use $a_{i}$, $A^\top$, $A^\dagger$, $\|A\|_2$, $\|A\|_F$, and $\mbox{Range}(A)$ to denote the $i$-th row, the transpose, the Moore-Penrose pseudoinverse, the spectral norm, the Frobenius norm, and the column space of $A$, respectively. For an integer $m\geq 1$, let $[m]:=\{1,\ldots,m\}$.
Given $\mathcal{T}\subset[m]$, the cardinality
of the set $\mathcal{T}$ is denoted by $| \mathcal{T}|$. We use  $A_{\mathcal{T},:}$ to denote the row  submatrix indexed by $\mathcal{T} $. We use  $\sigma_{\min}(A)$ to denote the smallest nonzero singular value of $A$.
Let $W$ be a $n\times n$ symmetric positive semidefinite matrix, we define $\|A\|_W=\sqrt{\operatorname{Tr}(A W A^\top)}$.

%

\subsection{Arbitrary sampling}

We analyze the performance of the proposed variance-reduced RK method under fully general data sampling strategies, known as arbitrary sampling.  
This general sampling paradigm, originally developed in \cite{richtarik2016parallel} for randomized coordinate descent algorithms, encompasses all proper random set-valued mappings $\mathcal{I}$ taking values in the collection of all subsets of $[m]$. The randomly sampled subset $\mathcal{I}$ follows an arbitrary proper probability distribution $\mathcal{D}$ defined over the $2^m$ subsets of $[m]$. In particular, a sampling is uniquely characterized by such a probability distribution over all subsets of $[m]$. A sampling is proper if each index $i \in [m]$ admits a positive sampling probability, i.e.,
$$p_i := \mathbb{P}(i\in \mathcal{I})>0, \ \ \forall  \ i\in[m].$$ 
Throughout this work, the term ``arbitrary sampling'' refers exclusively to proper samplings satisfying the above condition. We define 
$$
q_{ij}:=\mathbb{P}(i\in \mathcal{I},j\in \mathcal{I}).
$$
One can see that $q_{ii}=p_i$.
In many cases, the  $p_i$ and $q_{ij} $ admit closed-form expressions. We present two typical instances below.
\begin{itemize}
	\item[(1)] If $\mathcal{I}\sim \mathcal{D}$ is a uniformly random subset of size $s$ for a given minibatch size $s \in [m]$, then $p_i=q_{ii} = s/m$ for all $i\in[m]$, and $q_{ij}=\dfrac{s(s-1)}{m(m-1)}$ if $i\neq j$. 
	\item[(2)] Consider the following partition of $[m]$:
	$$	
	\begin{aligned}
		\mathcal{I}_i&=\left\{\varpi(k): k=(i-1)s+1,(i-1)s+2,\ldots,is\right\}, i=1, 2, \ldots, t-1,
		\\
		\mathcal{I}_t&=\left\{\varpi(k): k=(t-1)s+1,(t-1)s+2,\ldots,m\right\}, |\mathcal{I}_t|\leq s,
	\end{aligned}$$
	where $\varpi$ is a uniform random permutation on $[m]$ and $s$ is the block size. 
We first select an index $j \in [t]$ with probability proportional to $\|A_{\mathcal{I}_j,:}\|_F^2/\|A\|_F^2$. Then 
 \[
p_{i}=
\begin{cases}
	\frac{\|A_{\mathcal{I}_\ell,:}\|_F^2}{\|A\|_F^2}, & i\in  \mathcal{I}_\ell,\\
	0, & \text{ otherwise},
\end{cases} \ \ \text{and} \ \ 
q_{ij}=
	\begin{cases}
		\frac{\|A_{\mathcal{I}_\ell,:}\|_F^2}{\|A\|_F^2}, & i,j\in  \mathcal{I}_\ell,\\
		0, & \text{ otherwise}.
\end{cases}
\]
\end{itemize}
Furthermore, for any scalars $\theta_i$, with $i \in  [m]$, the following relations hold in expectation:
\begin{equation}
	\label{expectation-s}
	\begin{aligned}
		\mathbb{E}\left[ \sum_{i\in\mathcal{I}} \theta_i\right]&=\sum_{\mathcal{I}\subseteq[m]} \mathbb{P}(\mathcal{I})\left(\sum_{i\in\mathcal{I}} \theta_i\right)=\sum_{\mathcal{I}\subseteq[m]} \mathbb{P}(\mathcal{I})\left(\sum_{i\in[m]} \theta_i\mathbb{I}_{\mathcal{I}}(i)\right)\\
		&=\sum_{i\in[m]}\theta_i\sum_{\mathcal{I}\subseteq[m]} \mathbb{P}(\mathcal{I})\mathbb{I}_{\mathcal{I}}(i)=\sum_{i\in[m]}\theta_i\left(\sum_{\mathcal{I}: i\in\mathcal{I}} \mathbb{P}(\mathcal{I}) \right)=\sum_{i\in[m]} p_i\theta_i.
	\end{aligned}
\end{equation}
and
\begin{equation}
	\label{expectation-is}
	\mathbb{E}\left[ \left( \sum_{i\in\mathcal{I}} \theta_i\right)\left( \sum_{i\in\mathcal{I}} \theta_i\right)^\top\right]=\sum_{i\in[m]}\sum_{j\in[m]}q_{ij}\theta_i\theta_j.
\end{equation}

\subsection{Unbiased estimate and variance reduction }
In this subsection, we introduce two essential  concepts underlying the design and analysis of stochastic optimization algorithms, namely unbiased estimates and variance reduction.
Let $\mathcal{D}$ be an arbitrary proper probability distribution over subsets of $[m]$, with associated sampling probabilities $\{p_i\}_{i=1}^m$. We consider the finite-sum optimization problem
\begin{equation}
	\label{obj-func-o}
\min_{x} f(x) := \mathbb{E}_{\mathcal{I} \sim \mathcal{D}}\left[\sum_{i \in \mathcal{I}} f_i(x)\right] = \sum_{i=1}^m p_i f_i(x),
\end{equation}
where the equality follows from \eqref{expectation-s}. We say that a stochastic gradient $g^k$ is an \emph{unbiased estimate} of the full gradient $\nabla f(x^k)$ if
\begin{equation}
	\label{unbiased-est}
\mathbb{E}[g^k \mid x^k]=\nabla f(x^k).
\end{equation}
For instance, one can set $g^k = \sum_{i \in \mathcal{I}_k} \nabla f_i(x^k)$ with $\mathcal{I}_k \sim \mathcal{D}$.
The standard (minibatch) SGD update reads
\begin{equation}
	\label{SGD-iteration}
	x^{k+1} = x^k - \alpha_k g^k,
\end{equation}
where $\alpha_k$ denotes the stepsize. To guarantee convergence of \eqref{SGD-iteration} with constant stepsizes, the gradient estimator must possess vanishing variance near the optimal solution, i.e.,
\begin{equation}
	\label{VRP-xie}
	\mathbb{E}\left[\|g^k - \nabla f(x^k)\|^2_2\right] \to 0 \ \text{ as } \ x^k \to x^*,
\end{equation}
where the expectation is taken over all algorithmic randomness up to iteration $k$. We adopt \eqref{VRP-xie} as the defining property of \emph{variance reduction} (VR) methods.


\subsection{Inconsistency leads to variance}
 \label{section-in-variance}
 
This subsection aims to provide a brief explanation of the variance issue introduced by the inconsistency of the linear system \eqref{main-prob}.
Consider the following  least‑squares problem 
\begin{equation}\label{quar-f}  
	\min f(x) = \frac{1}{2} \|Ax - b\|_2^2   = \sum_{i=1}^m p_i f_i(x),
\end{equation}
where 
\( f_i(x) := \frac{( \langle a_{i},x\rangle - b_i)^2}{2p_i} \).  
By choosing \( p_i = \frac{\|a_{i}\|_2^2}{\|A\|_F^2} \) and \( \alpha_k = \frac{1}{\|A\|_F^2} \), the SGD iteration \eqref{SGD-iteration} reduces to the RK method
\[
x^{k+1}= x^k-\alpha_k\nabla f_{i_k}(x^k)= x^k - \frac{\langle a_{i_k,:}, x^k\rangle - b_{i_k}}{\|a_{i_k}\|_2^2} a_{i_k}.
\]
Let   \( r^* := AA^\dagger b - b \).  
Evaluating the VR property \eqref{VRP-xie} at $x^*=A^\dagger b$ gives
\[
\mathbb{E}\left[\| \nabla f_{i_k}(A^\dagger b) - \nabla f(A^\dagger b) \|_2^2\right] = \mathbb{E}\left[\| \nabla f_{i_k}(A^\dagger b) \|_2^2\right] = \sum_{i=1}^m  \|A\|^2_F (\langle a_{i}, A^\dagger b\rangle - b_i)^2= \|A\|^2_F\|r^*\|^2_2.
\]  
Then, when the system is consistent, the iterates converge to a solution that satisfies \( A A^\dagger b = b \), implying \( r^* = 0 \), and hence the variance vanishes. This enables standard SGD (and RK) to converge even with constant stepsizes.  In the inconsistent case, however,  \(r^* \neq 0\) and the variance remains strictly positive. This persistent variance prevents exact convergence and motivates the development of variance reduction techniques.

\section{Variance-reduced RK: A residual-tracking control variate approach}

In this section, we present the motivation and design of our variance-reduced RK method. To address the variance caused by inconsistency, we first introduce an idealized but impractical method, RK$^*$, which eliminates variance using the optimal residual $r^* = AA^\dagger b - b$. Based on this ideal, we develop practical approximations to progressively approach $r^*$, thereby reducing variance in realistic settings. We refer to our approach as \emph{residual-tracking control variates}.


\subsection{RK$^*$}

One easy fix that makes the RK method converge is to simply shift each projection vector by \(r^*_{i_k} \). That is, we propose the following RK$^*$ method
\begin{equation}
	\label{RK-star}
x^{k+1}
= x^k -  \left( \frac{\langle a_{i_k}, x^k\rangle - b_{i_k}}{\|a_{i_k}\|_2^2}  -\frac{r^*_{i_k}}{\|a_{i_k}\|_2^2}  \right)a_{i_k}, 
\end{equation}
where $\mathbb{P}(i_k=i)=\frac{\|a_{i}\|_2^2}{\|A\|_F^2}$.
We note that RK$^*$ is practically infeasible, as $r^*$ is unknown. Nevertheless, it provides the key insights motivating our approach, suggesting that if we can approximate $r^*$ adaptively, the desired convergence behavior may be preserved.

First, RK$^*$ uses an unbiased estimate of the full gradient. Indeed, since $A^\top r^*=0$, the gradient estimate is unbiased
$$
\mathbb{E}\left[ \left( \frac{ \langle a_{i_k}, x^k\rangle - b_{i_k}}{\|a_{i_k}\|_2^2}  -\frac{r^*_{i_k}}{\|a_{i_k}\|_2^2}  \right) a_{i_k} \right]=\frac{A^\top(Ax^k-b-r^*)}{\|A\|^2_F}=\frac{A^\top(Ax^k-b)}{\|A\|^2_F}.
$$
Moreover,  RK$^*$ satisfies the VR property \eqref{VRP-xie} as $x^k$ approaches $A^\dagger b$. Indeed,
\[
\mathbb{E}\left[\|g^k - \nabla f(x^k)\|_2^2\right] = \mathbb{E}\left[\|\nabla f_i(x^k) - \nabla f_i(A^\dagger b) - \nabla f(x^k)\|_2^2\right] \le \mathbb{E}\left[\|\nabla f_i(x^k) - \nabla f_i(A^\dagger b)\|_2^2\right],
\]
where we used the inequality $
\mathbb{E}\left[\|X - \mathbb{E}[X]\|_2^2\right] \le \mathbb{E}\left[\|X\|_2^2\right]
$ with $X = \nabla f_i(x^k) - \nabla f_i(x^*)$ and the fact that $\mathbb{E}[\nabla f_i(x^k) - \nabla f_i(x^*)] = \nabla f(x^k)$.

In fact, RK$^*$ enjoys the same convergence guarantee as the standard RK method, stated below.
\begin{theorem}\label{prop-rk-star}
	Let $\{x^k\}_{k\geq0}$ be the sequence generated by RK$^*$ \eqref{RK-star} with initial point $x^0\in\operatorname{Range}(A^\top)$. Then
	$$
	\mathbb{E}[\|x^k-A^\dagger b\|^2_2]\leq\left(1-\frac{\sigma_{\min}^2(A)}{\|A\|^2_F}\right)^k\|x^0-A^\dagger b\|^2_2.
	$$
\end{theorem}


\subsection{Residual-tracking control variates }

Inspired by  RK$^*$, we develop a practical convergent RK-type method equipped with a dynamically updated residual sequence $\bar{r}^k$ to approximate the inaccessible optimal residual $r^*$. Unlike the fixed optimal residual adopted in RK$^*$, the sequence $\bar{r}^k$ is iteratively updated and computationally feasible. Furthermore, the constructed $\bar{r}^k$ is designed to satisfy two core properties: the unbiasedness condition \eqref{unbiased-est} and the VR property \eqref{VRP-xie}.

Recall that $\mathcal{D}$ denotes an arbitrary proper probability distribution over subsets of $[m]$ with sampling probabilities $\{p_i\}_{i=1}^m$\footnote{For first-time readers, we recommend following the subsequent arguments by setting $p_i=1/m$ and $\mathcal{I}=\{i\}$ with probability $p_i$.}. Let $v_i\in\mathbb{R}$ denote scalar auxiliary variables for $i=1,\dots,m$. We reformulate the full residual as
\[
\begin{aligned}
	r(x):=Ax-b
	=\sum_{i=1}^{m} p_i \left(\frac{r_i(x)}{p_i}e_i-\frac{v_i}{p_i} e_i+\frac{v_i}{ p_i}e_i\right)
	=\sum_{i=1}^{m} p_i\left(\frac{r_i(x)}{p_i}e_i-\frac{v_i}{ p_i} e_i\right)+ \sum_{j=1}^mv_je_j.
\end{aligned}
\]
Let us define the component-wise residual vector
\[
r_i(x,v):=\frac{r_i(x)}{p_i}e_i-\frac{v_i}{p_i} e_i, \quad i=1,\dots,m.
\]
Based on this definition, we construct an unbiased residual estimator $\sum_{i\in\mathcal{I}} r_i(x,v)+\sum_{j=1}^mv_je_j$ via sampling $\mathcal{I}\sim \mathcal{D}$. This further yields an unbiased estimator of the full gradient, which satisfies
\begin{equation}\label{gradient-ue}
		\mathbb{E} \left[\sum_{i\in\mathcal{I}} A^\top r_i(x,v)+ A^\top\sum_{j=1}^mv_je_j\right]
		=A^\top r(x)
		=\nabla f(x).
\end{equation}
We next characterize the variance of the corresponding stochastic gradient estimator
$g^k = \sum_{i\in\mathcal{I}_k}A^\top r_i(x^k,v)+A^\top \sum_{j=1}^mv_je_j$. Substituting this in \eqref{VRP-xie}  gives
\begin{equation}\label{gradient-VR}
	\begin{aligned}
		\mathbb{E}\left[\|g^k - \nabla f(x^k)\|_2^2\right]
		&= \mathbb{E} \left[ \left\|A^\top \sum_{i\in\mathcal{I}}r_i(x^k,v)+ A^\top\sum_{j=1}^mv_je_j-A^\top r(x^k) \right\|_2^2 \right]\\
		&\leq \mathbb{E} \left[ \left\|A^\top \sum_{i\in\mathcal{I}} \left( \frac{r_i(x^k)}{p_i}e_i-\frac{v_i}{p_i} e_i\right)\right \|_2^2 \right]\\
		&\leq \|A\|^2_2\mathbb{E} \left[ \left\| \sum_{i\in\mathcal{I}} \left( \frac{r_i(x^k)}{p_i}e_i-\frac{v_i}{p_i} e_i\right)\right \|_2^2 \right]\\
		&=  \|A\|^2_2\mathbb{E} \left[ \sum_{i\in\mathcal{I}} \frac{1 }{p_i^2}\left(r_i(x^k)-v_i\right)^2 \right]\\
		&=  \|A\|^2_2 \sum_{i=1}^{m}  \frac{1}{p_i} \left(r_i(x^k)-v_i\right)^2,
	\end{aligned}
\end{equation}
where the first inequality follows from  the  variance bound $
\mathbb{E}\left[\|X - \mathbb{E}[X]\|_2^2\right] \le \mathbb{E}\left[\|X\|_2^2\right]$ applied to $X = A^\top \sum_{i\in\mathcal{I}}\left(\frac{r_i(x)}{p_i}e_i-\frac{v_i}{p_i} e_i \right)$, and the last equality is a direct consequence of \eqref{expectation-s}. This bound implies that the estimator achieves variance reduction whenever $v_i$ approximates the optimal residual $r^*_i$ as $x^k\to x^*$, which justifies the interpretation of $\{v_i\}_{i=1}^m$ as control variates.

The auxiliary variables $\{v_i\}_{i=1}^m$ are flexible and can be optimized to minimize the estimation variance. The RK$^*$ method adopts the ideal choice $v_i = r^*_i$, which is theoretically valid but computationally impractical since $r^*_i$ is unknown in general. To address this issue, we adopt a feasible alternative by evaluating residuals at dynamically updated reference points. For each residual component $r_i$, we maintain a reference point $\bar{x}^i\in\mathbb{R}^d$ and set $v_i = r_i(\bar{x}^i)$, where $\bar{x}^i$ denotes the \emph{most recent iterate} at which $r_i(\cdot)$ is evaluated. 
Substituting this control variate yields the residual estimator
\[
	\sum_{i\in\mathcal{I}_k}r_{i}(x^k,v^k)+ \sum_{j=1}^m r_j(\bar{x}^j)e_j
	=\sum_{i\in\mathcal{I}_k}\left( \frac{r_{i}(x^k)}{p_{i}} e_{i} -\frac{r_{i}(\bar{x}^{i})}{p_{i}} e_{i}\right)  +  \sum_{j=1}^m r_j(\bar{x}^j)e_j. 
\]
Since each \(\bar{x}^i\) is updated whenever \(i\in\mathcal{I}_k\), the stored residual values vary across iterations. We collect the values available at the beginning of iteration $k$ into the vector $\bar r^k\in\mathbb R^m$, defined by 
\(
(\bar{r}^k)_i:=r_i(\bar{x}^i)
\), and set $v^k=\bar r^k$.
Accordingly, the residual estimator can be rewritten as
\begin{equation}
	\label{re-update}
	\sum_{i\in\mathcal{I}_k}r_i(x^k,v^k)+\sum_{j=1}^{m} (\bar{r}^{k})_j e_j =\sum_{i\in\mathcal{I}_k}\frac{1}{p_{i}}\left( r_{i}(x^k)  -(\bar{r}^{k})_i \right)e_{i} +   \sum_{j=1}^{m}  (\bar{r}^{k})_j e_j.
\end{equation}
It follows that the sequence $\{\bar{r}^k\}_{k \ge 0}$ follows the update rule:
\begin{equation}
	\label{barr-update}
(\bar{r}^{k+1})_i =
\begin{cases}
	\langle a_i, x^k \rangle - b_i, & i \in \mathcal{I}_k, \\[4pt]
	(\bar{r}^{k})_i, & \text{otherwise}.
\end{cases}
\end{equation}
We term the updated sequence $\bar{r}^{k}$ the \emph{residual-tracking control variate}. 
By construction, $\bar{r}^k$ satisfies both the unbiasedness condition \eqref{gradient-ue} and the VR property \eqref{gradient-VR} if $\bar{r}^k \to r^*$. This strategy indeed works, as the variance of $g^k$ diminishes to zero (see, e.g., Lemma \ref{lemma-2}).

\subsection{Variance-reduced  RK with arbitrary sampling}

Building upon the residual-tracking control variates $\bar{r}^k$ introduced in the previous subsection, we now present the variance-reduced RK with arbitrary sampling (VRRK-AS). 
From \eqref{re-update}, we know that the unbiased estimator $g^k$ of the full gradient $\nabla f(x^k)$ now become
$$
\begin{aligned}
	g^k &= 
	A^\top \sum_{i\in\mathcal{I}_k} \frac{1}{p_{i}}\left( r_{i}(x^k)  -(\bar{r}^{k})_i \right)e_{i}+A^\top \sum_{j=1}^{m} (\bar{r}^{k})_j e_j\\
	&= \sum_{i\in\mathcal{I}_k} \frac{1}{p_{i}}\left( \langle a_i,x^k\rangle-b_i  -(\bar{r}^{k})_i \right)a_{i}+ \sum_{j=1}^{m} (\bar{r}^{k})_j a_j
\end{aligned}
$$
We note that the term $\sum_{i=1}^m (\bar{r}^k)_i a_i$ can be computed efficiently. Indeed, let $e^k:= \sum_{i=1}^{m}(\bar{r}^{k})_ia_{i}$, then we have
$$
\begin{aligned}
	e^{k+1}&=\sum_{i=1}^{m}(\bar{r}^{k+1})_ia_{i}=\sum_{i\notin\mathcal{I}_k}(\bar{r}^{k})_ia_{i}+\sum_{i\in\mathcal{I}_k}  (\langle a_i,x^k\rangle-b_i)a_i\\
	&=e^{k}+\sum_{i\in\mathcal{I}_k} \left( \langle a_i,x^k\rangle-b_i- (\bar{r}^{k})_i\right)a_{i},
\end{aligned}
$$
where the second equality follows from the iteration scheme of $\bar{r}^{k+1}$ in \eqref{barr-update}.

	Now we are ready to formally state the VRRK-AS method in Algorithm \ref{RABK-VR-EFF}. It can be seen that the per-iteration computational cost of the Algorithm  is $O(|\mathcal{I}_k| n)$, dominated by the computation of the inner products $\langle a_i, x^k \rangle$ for $i \in \mathcal{I}_k$ and the subsequent vector updates. In terms of storage, the algorithm requires maintaining three vectors, namely the iterate $x^k \in \mathbb{R}^n$, the auxiliary vector $e^k \in \mathbb{R}^n$, and the residual estimate $\bar{r}^k \in \mathbb{R}^m$.  Thus the total storage is $O(n + m)$.

\begin{algorithm}[htpb]
	\caption{Variance-reduced  RK with arbitrary sampling (VRRK-AS)   \label{RABK-VR-EFF}}
	\begin{algorithmic}
		\Require
		$A\in \mathbb{R}^{m\times n}$, $b\in \mathbb{R}^m$, distribution $\mathcal{D}$ with sampling probabilities $\{p_i\}_{i=1}^m$, stepsize $\alpha$, $k=0$, initial points $x^0=0$, $e^0=-\sum_{i=1}^{m}b_i a_i$, and $\bar{r}^{0}=\left( -b_1,\ldots, -b_m\right)$.
		\begin{enumerate}
			\item[1:] Draw sample $\mathcal{I}_k\sim \mathcal{D}$.
			\item[2:] Compute
			$$
			g^k=e^{k}+\sum_{i\in\mathcal{I}_k} \frac{1}{p_i}\left(\langle a_i,x^k\rangle-b_i-(\bar{r}^{k})_i\right) a_{i}.
			$$
			\item[3:] Update
			$$
			\begin{aligned}
				&x^{k+1}=x^k-\alpha g^k,
					\\
					&	e^{k+1}=e^{k}+\sum_{i\in\mathcal{I}_k} \left( \langle a_i,x^k\rangle-b_i- (\bar{r}^{k})_i\right)a_{i},\\
					&(\bar{r}^{k+1})_i=\begin{cases}
						\langle a_{i},x^k\rangle-b_{i},\ &\text{if} \ i\in \mathcal{I}_k,
						\\
						(\bar{r}^{k})_i, \ &\text{otherwise}.
					\end{cases}
				\end{aligned}
				$$
				\item[4:] If the stopping rule is satisfied, stop and go to output. Otherwise, set $k=k+1$ and return to Step $1$.
			\end{enumerate}
			
			\Ensure
			The approximate solution.
		\end{algorithmic}
	\end{algorithm}

Next, we state the linear convergence result for Algorithm \ref{RABK-VR-EFF}. For convenience, we define
\begin{equation}
	\label{def-D}
	p:=\left( p_1,\ldots,p_m\right)^\top, \ P:=\diag\left( p\right),\  Q:=(q_{ij}), \ \text{and} \ D:=\diag\left( \frac{\|a_1\|^2_2}{p_1},\ldots,\frac{\|a_m\|^2_2}{p_m}\right).
\end{equation}
By the definitions of $p_i$ and $q_{ij}$, we know that $Q - pp^\top$ is the covariance matrix of the random vector $\sum_{i \in \mathcal{I}} e_i $. Hence $Q - pp^\top$ is positive semidefinite.

\begin{theorem}
	\label{main-thm-linear}
	For any given linear system $Ax = b$, let $\{x^k\}_{k\geq 0}$ and $\{\bar{r}^k\}_{k\geq 0}$ be the iteration sequences generated by Algorithm \ref{RABK-VR-EFF}. Let $r^*=A A^\dagger b- b$, $\beta \in\left(0,\frac{1}{\lambda_{\max}(D)}\right]$ be any constant,  and define
	$
	\Phi^k:=\|x^{k}-A^\dagger b\|^2_2 +\alpha\beta   \|\bar{r}^k -  r^*\|_{DP^{-1}}^2.
	$ If the stepsize $\alpha$ satisfies
	$$
	\alpha\leq \min\left\{\frac{1-\beta\lambda_{\max}(D)}{2\lambda_{\max}(D)\lambda_{\max}(QP^{-1})}, \frac{\beta\lambda_{\min}(P)}{2\lambda_{\max}(Q-pp^\top)+\beta\sigma_{\min}^2(A)}  \right\}
	$$
	with $p$, $P$, $Q$, and $D$ being defined as \eqref{def-D}, then 
	$$
	\mathbb{E}[\Phi^k] \leq\left( 1-\alpha\sigma_{\min}^2(A)\right)^k \Phi^0.
	$$
\end{theorem}

\begin{remark}\label{factor-comp}
	Consider the probability sampling rule adopted in standard RK methods, where $p_i = \frac{\|a_{i}\|_2^2}{\|A\|_F^2}$. 
	Indeed, under this setting we have $Q = P$, $D = \|A\|_F^2 I$, $\lambda_{\max}(D) = \|A\|_F^2$, $\lambda_{\min}(P) = \frac{\min_i \|a_i\|_2^2}{\|A\|_F^2}$, and $\lambda_{\max}(Q - pp^\top) \le \frac{\max_i \|a_i\|_2^2}{\|A\|_F^2}$. 
	Substituting these quantities into Theorem \ref{main-thm-linear}, we find that the choices 
	$
	\beta = \frac{1 - 2c}{\|A\|_F^2} \quad \text{and} \quad \alpha = \frac{c}{\|A\|_F^2}
	$
	satisfy the conditions of the theorem, where 
	$
	c = \frac{\min_i \|a_i\|_2^2}{2 \max_i \|a_i\|_2^2+2 \min_i \|a_i\|_2^2 + \sigma_{\min}^2(A)}.
	$
	Under this choice, the convergence factor of Algorithm \ref{RABK-VR-EFF} in Theorem \ref{main-thm-linear} becomes 
	\[
	1 - c \cdot \frac{\sigma_{\min}^2(A)}{\|A\|_F^2},
	\]
	which is of the same order as the convergence factor of the classical RK method.
\end{remark}

%

\subsection{Connection with the Jacobian sketching and SAGA}
\label{section-34}
Gower et al. \cite{gower2021stochastic} proposed a novel family of variance-reduced SGD methods, called JacSketch, which constructs accurate gradient estimators via a Jacobian sketching technique. 
For the unconstrained optimization problem \eqref{obj-func-o} and the least-squares problem \eqref{quar-f}, we define the objective vector and Jacobian matrix as
\[
F(x):=(f_1(x),\ldots,f_m(x))^{\top}, \qquad 
G(x):=[\nabla f_1(x),\ldots, \nabla f_m(x)].
\]
Starting from an arbitrary initialization $J^0 \in \mathbb{R}^{n \times m}$, JacSketch iteratively maintains a sequence $\{ {J}^k\}_{k\geq 0}$ that approximates the true Jacobian $G(x^k)$. At each iteration, the estimator is updated via a sketch-and-project rule \cite{Gow15}:
\[
J^{k+1} = \arg \min_{ {J} \in \mathbb{R}^{n \times m}} \|  {J} -  {J}^k \|^2_F 
\quad \text{s.t.} \quad 
J  {S}_k =  {G}(x^k)  {S}_k,
\]
where the random sketching matrix $ {S}_k$ is independently sampled from a predefined distribution $\mathcal{D}$. This update yields the closest matrix to $ {J}^k$ that matches the true Jacobian action over the sketched subspace spanned by $ {S}_k$.

In particular, if $ {S}_k =\sum_{i\in\mathcal{I}_k}e_{i} e_{i}^\top$, then
$
 {J}^{k+1}= {J}^k+\big( {G}(x^k)- {J}^k\big)\sum_{i\in\mathcal{I}_k}e_{i} e_{i}^\top.
$
Note that
$ {F}(x)^\top p=f(x)$ and $  {G}(x) p=\nabla f(x)$. To ensure that \(\tilde{g}^k\) is an unbiased estimator of the full gradient \(\nabla f(x^k)\), we set
$
\tilde{g}^k =  {J}^k p + ( {G}(x^k) -  {J}^k) \sum_{i \in \mathcal{I}_k} e_i.
$

Let \(x^0 = 0\) and \( {J}^0 =  {G}(x^0) = A^\top \operatorname{diag}(\bar{r}^0) P^{-1}\). Then JacSketch follows the iterative scheme:
\begin{equation}
	\label{JacSketch}
	\begin{cases}
		\begin{aligned}
			&\tilde{g}^k= {J}^k p+\big( {G}(x^k)- {J}^k\big)\sum_{i\in\mathcal{I}_k} e_{i},\\
			&x^{k+1}=x^k-\alpha \tilde{g}^k,\\
			& {J}^{k+1}= {J}^k+\big( {G}(x^k)- {J}^k\big)\sum_{i\in\mathcal{I}_k}e_{i} e_{i}^\top.
		\end{aligned}
	\end{cases}
\end{equation}

We have the following equivalence result.

\begin{prop}\label{THM-EQ}
	If Algorithm \ref{RABK-VR-EFF} and the JacSketch update \eqref{JacSketch} for \eqref{quar-f} use identical sampling indices $\{\mathcal{I}_k\}_{k\ge 0}$ and stepsize $\alpha$, their iterate sequences $\{x^k\}_{k\ge 0}$ coincide exactly. Therefore, Theorem~\ref{main-thm-linear} provides stronger convergence guarantees for JacSketch applied to \eqref{quar-f} than previously available \cite[Theorem 1]{gower2021stochastic}.
\end{prop}
The above proposition bridges our algorithm and JacSketch. As SAGA is a special case of JacSketch \cite[Section 1.4]{gower2021stochastic}, our method is inherently related to SAGA as well. Despite this connection, our convergence analysis explicitly allows rank-deficient matrices, for which the least-squares problem \eqref{quar-f} fails to be strongly convex. We prove linear convergence of our algorithm for arbitrary matrix $A$. Proposition~\ref{THM-EQ} shows that VRRK-AS is a residual-coordinate
representation of the corresponding JacSketch iteration for
least-squares problems, rather than an unrelated stochastic method.
Although a rank-deficient least-squares objective is not strongly
convex on the entire space, it is strongly convex when restricted
to \(\operatorname{Range}(A^\top)\).
Our analysis works directly in the original variables and explicitly
tracks the residual-estimation error and the minimum-norm solution
selected by the initialization. This viewpoint makes the dependence on the sampling distribution
and the matrix structure explicit within the row-action framework.


\section{Accelerating VRRK}

In this section, we develop an accelerated algorithm for solving the inconsistent linear system \eqref{main-prob}. Inspired by the success of existing accelerated variance-reduced optimization frameworks \cite{allen2018katyusha,kovalev2020don,mishchenko2025distributed,pmlr-v119-li20g}, we incorporate effective acceleration mechanisms into the original VRRK-AS method to further improve its convergence performance. The resulting novel AVRRK-AS method, outlined in Algorithm \ref{compression-RABK-EF}, integrates momentum-based acceleration and variance reduction, and serves as an accelerated extension of Algorithm \ref{RABK-VR-EFF}.

			\begin{algorithm}[htpb]
		\caption{Accelerating VRRK-AS (AVRRK-AS)  \label{compression-RABK-EF}}
		\begin{algorithmic}
			\Require
			$A\in \mathbb{R}^{m\times n}$, $b\in \mathbb{R}^m$, distribution $\mathcal{D}$ with sampling probabilities $\{p_i\}_{i=1}^m$,  parameters $\alpha$, $\theta_1$, $\theta_2$, $\eta$,  $\beta$, and $\gamma$, $k=0$, initial points $x^0=z^0=y^0=y^{-1}=0$, $e^0=-\sum_{i=1}^{m}b_i a_i$, and $\bar{r}^{0}=\left( -b_1,\ldots, -b_m\right)$.
			\begin{enumerate}
				\item[1:] Compute $x^{k}=\theta_1 z^k+\theta_2 y^{k-1}+(1-\theta_1-\theta_2)y^k$. 
				\item[2:] Draw sample $\mathcal{I}_k\sim \mathcal{D}$.
				\item[3:] Compute
				$$
				g^k=e^{k}+\sum_{i\in\mathcal{I}_k} \frac{1}{p_i}\left(\langle a_i,x^k\rangle-b_i-(\bar{r}^{k})_i\right) a_{i}.
				$$
				\item[4:] Update
				$$
				\begin{aligned}
					&y^{k+1}=x^k-\alpha g^k,\\
					&z^{k+1}=\beta z^{k}+(1-\beta)x^k+\frac{\gamma}{\alpha} (y^{k+1}-x^k),\\
					&	e^{k+1}=e^{k}+\eta\sum_{i\in\mathcal{I}_k} \left( \langle a_i,y^{k-1}\rangle-b_i- (\bar{r}^{k})_i\right)a_{i},\\
						&(\bar{r}^{k+1})_i=\begin{cases}
						\eta(\langle a_{i},y^{k-1}\rangle-b_{i})+(1-\eta)(\bar{r}^{k})_i,\ &\text{if} \ i\in \mathcal{I}_k,
						\\
						(\bar{r}^{k})_i, \ &\text{otherwise}.
					\end{cases}
				\end{aligned}
				$$
				\item[5:] If the stopping rule is satisfied, stop and go to output. Otherwise, set $k=k+1$ and return to Step $1$.
			\end{enumerate}
			
			\Ensure
			The approximate solution.
		\end{algorithmic}
	\end{algorithm}
	
	It can be seen that AVRRK-AS achieves a per-iteration complexity of $O(|\mathcal{I}_k| n)$, which matches the VRRK-AS method.  The acceleration and residual-tracking mechanisms introduce no extra per-iteration computational cost.
	For the storage complexity,  AVRRK-AS requires storing four iterative vectors $x^k, y^k, y^{k-1}, z^k\in\mathbb{R}^d$, one auxiliary vector $e^k\in\mathbb{R}^d$, and the residual-tracking vector $\bar{r}^k\in\mathbb{R}^m$.  Compared to VRRK-AS, which stores three vectors $x^k$, $e^k$, $\bar{r}^k$, AVRRK-AS requires three additional vectors. Nevertheless, all stored variables are lightweight vector quantities, leading to a low and acceptable memory overhead for large-scale tasks.
	
	Using a similar argument as in Subsection \ref{section-34}, Algorithm \ref{compression-RABK-EF} can admit the following JacSketch formulation:
	\begin{equation}
		\label{Acc-JacSketch}
		\begin{cases}
			\begin{aligned}
				&x^{k}=\theta_1 z^k+\theta_2 y^{k-1}+(1-\theta_1-\theta_2)y^k\\
				&g^k= {J}^k p+( {G}(x^{k})- {J}^k)\sum_{i\in\mathcal{I}_k} e_{i},\\
				&	y^{k+1}=x^k-\alpha g^k,\\
				&z^{k+1}=\beta z^{k}+(1-\beta)x^k+\frac{\gamma}{\alpha} (y^{k+1}-x^k),\\
				& {J}^{k+1}= {J}^k+ \eta( {G}(y^{k-1})- {J}^k)\sum_{i\in\mathcal{I}_k}e_{i} e_{i}^\top,
			\end{aligned}
		\end{cases}
	\end{equation}
	where $x^0=z^0=y^0=y^{-1}=0$, and $ {J}^0= {G}(x^0)$.
	This representation connects the proposed acceleration mechanism
	with the JacSketch formulation of least-squares problems.
	Our focus is on combining momentum with residual tracking while
	preserving row-only access, and on deriving a convergence guarantee
	for the resulting scheme.
\subsection{Convergence analysis}

We use the following notations
$$
\mathcal{Z}^{k}:=\|z^k-A^\dagger b\|^2_2, \ \ \mathcal{Y}^{k}:=\|Ay^k-b\|^2_2-\|Ax^*-b\|^2_2,
$$
and the Lyapunov function
$$
\Psi^{k+1}:= \mathcal{Z}^{k+1} +\frac{ \gamma\beta }{\theta_1} \mathcal{Y}^{k+1}  +\frac{\gamma\beta \theta_2(1+\theta_1)}{\theta_1} \mathcal{Y}^{k}+\frac{8\alpha \gamma\lambda_{\max}(Q-pp^\top) }{\theta_1 \eta\lambda_{\min}(P)}\left\| {J}^{k+1}- {G}(y^{k  })\right\|^2_{F},
$$
where $Q$ and $p$ are defined as \eqref{def-D}.

\begin{theorem}
	\label{thm-acc}
	Let the parameters $\alpha$, $\theta_1$, $\theta_2$, $\eta$,  $\beta$, and $\gamma$ in Algorithm \ref{compression-RABK-EF} be chosen as 
	$$
	\begin{aligned}
		 \alpha&=\min\left\{\frac{(1-\eta)^2}{128\lambda_{\max}(DP^{-1}) \lambda_{\max}(Q-pp^\top)(3-\eta) }, \ \ \frac{1}{4\|A\|^2_2}\right\}, \ \  \theta_1=\min\left\{\frac{1}{4},\sqrt{\frac{\alpha\sigma_{\min}^2(A)}{2}} \right\},\\
		  \theta_2&=\frac{1}{2}, \ \ \gamma=\frac{\alpha}{2(\theta_1+\alpha\sigma_{\min}^2(A))}, \ \ \eta=\frac{1}{\lambda_{\min}(P)+1}, \ \ \beta=1-\gamma\sigma_{\min}^2(A), 
	\end{aligned}
	$$
	where the matrices $Q$, $P$, and $D$ are defined as \eqref{def-D}.
	 Then the iteration sequence generated by Algorithm \ref{compression-RABK-EF} satisfies
	$$
	\mathbb{E}\left[\Psi^{k}\right]\leq   \left(1-\min\left\{ \frac{\eta\lambda_{\min} (P)}{4},\frac{\sqrt{\alpha\sigma^2_{\min}(A)}}{4}  \right\}\right)^k\Psi^0.
	$$
\end{theorem}

\begin{remark}\label{remark-quadratic} 
	It follows from Theorem~\ref{main-thm-linear} and Theorem~\ref{thm-acc} that the VRRK-AS and AVRRK-AS methods have iteration complexities
	$$
	O\left(
	\frac{1}{\alpha\sigma_{\min}^2(A)}
	\log\frac1\varepsilon
	\right)
	  \ \text{and} \
	O\left(
	\max\left\{
	\frac{1}{\eta\lambda_{\min}(P)},
	\frac{1}{\sqrt{\alpha \sigma_{\min}^2(A)}}
	\right\}
	\log\frac1\varepsilon
	\right), 
	$$
	respectively.
	%
The appearance of $\sqrt{\alpha \sigma_{\min}^2(A)}$ instead of $\alpha \sigma_{\min}^2(A)$ indicates a quadratic improvement, making the accelerated method substantially faster for ill-conditioned problems.
\end{remark}

\begin{remark}
	If $\lambda_{\max}(Q-pp^\top)=0$, then $Q-pp^\top=0$,
	and the gradient estimator has zero conditional variance.
	In this case, choose
	$
	\alpha=\frac{1}{4\|A\|_2^2}
	$
	and retain the other parameter choices in
	Theorem~\ref{thm-acc}.
	Now the Lyapunov function reduces to
	$
	\Psi^k=
	\mathcal Z^k
	+\frac{\gamma\beta}{\theta_1}\mathcal Y^k
	+\frac{\gamma\beta\theta_2(1+\theta_1)}
	{\theta_1}\mathcal Y^{k-1}.
	$
	The same argument as in the proof of
	Theorem~\ref{thm-acc}, with the variance and memory
	terms omitted, yields
	\[
	\mathbb E[\Psi^k]
	\leq
	\left(
	1-\frac{\sqrt{\alpha\sigma_{\min}^2(A)}}{4}
	\right)^k\Psi^0.
	\] 
	Thus, the iteration complexity is \(O\bigl(\frac{1}{\sqrt{\alpha\sigma_{\min}^2(A)}}\log\frac{1}{\varepsilon}\bigr)\). Compared with the \(O\bigl(\frac{1}{\alpha\sigma_{\min}^2(A)}\log\frac{1}{\varepsilon}\bigr)\) rate of the nonaccelerated bound, this yields a quadratic improvement.
\end{remark}

\section{Numerical experiments}
\label{sec:numerical}

In this section, we investigate the numerical performance of VRRK-AS
(Algorithm~\ref{RABK-VR-EFF}) and AVRRK-AS
(Algorithm~\ref{compression-RABK-EF}) on inconsistent linear systems. We
first compare the two proposed methods to identify when acceleration
compensates for its additional per-iteration cost. We then compare them
with REK~\eqref{REK} and RCD. All algorithms are implemented in {\sc Matlab}
R2022a and run under Windows 11 on a desktop computer equipped with an
Intel Core i7-1360P processor and 32 GB of RAM. The code to reproduce our results can be found at  \href{https://github.com/xiejx-math/VRRK}{https://github.com/xiejx-math/VRRK}.

\subsection{Numerical setup}
\label{NS-Section61}

We consider synthetic matrices and real-world sparse matrices. For given
$m$, $n$, $r$, and $\kappa>1$, a synthetic matrix is generated as
$
 A=UDV^\top,
$
where $U\in\mathbb{R}^{m\times r}$ and
$V\in\mathbb{R}^{n\times r}$ have orthonormal columns obtained from thin
QR factorizations of independent standard Gaussian matrices, and
$
 D=\diag(\sigma_1,\ldots,\sigma_r)
 $ with $
 \sigma_j=\kappa^{-(j-1)/(r-1)}.
$
Thus, $\rank(A)=r$, and the ratio of its largest to smallest nonzero
singular values is $\kappa$. The real-world matrices are taken from the
SuiteSparse Matrix Collection~\cite{Kol19}. 

For the synthetic matrix, we first generate a solution
$x^*=Vc$, where $c\sim\mathcal{N}(0,I_r)$, and normalize it so that
$\|x^*\|_2=1$. We then draw $z\sim\mathcal{N}(0,I_m)$ and define
\begin{equation}
	 b=Ax^*+b_e \ \ \text{with} \ \ b_e=
	 \delta\frac{\|Ax^*\|_2}{\|\title{b}\|_2}\title{b} \ \ \text{and} \ \  \title{b}=(I-UU^\top)z.
 \label{eq:inconsistent-rhs}
\end{equation}
Here $\delta>0$ is the relative inconsistency level. Since
$A^\top b_e=0$, we have $A^\dagger b=x^*$ and $r^*=AA^\dagger b-b=-b_e$. In addition to this
\emph{random} solution, we consider the more challenging
\emph{tail-aligned} solution
\begin{equation}
 x^*=\frac{1}{\sqrt q}\sum_{j=r-q+1}^{r}v_j, 
 \label{eq:tail-aligned-solution}
\end{equation}
where $v_j$ is the $j$-th right singular vector of $A$ and $q$ is a positive integer. This construction
concentrates the initial error in the subspace associated with the
smallest positive singular values and therefore emphasizes the slow
spectral components for which acceleration is expected to be most useful. 
For a SuiteSparse matrix, we draw and normalize
$x^*\in\mathbb{R}^n$, generate $z\sim\mathcal{N}(0,I_m)$, and compute
$
 b_e=z-A A^\dagger z.
$
After scaling $b_e$ as in \eqref{eq:inconsistent-rhs}, we set
$b=Ax^*+b_e$. This also gives $A^\top b_e=0$.

All methods are initialized with \(x^0=0\). For VRRK-AS and AVRRK-AS, we set
$
\bar r^0=-b$ and $
e^0=-\sum_{i=1}^{m}b_i a_i,
$
whereas REK is initialized with \(z^0=b\). All parameters are tuned on problem instances generated independently of the final test set. For VRRK-AS, we parameterize the stepsize as
$
\alpha=\frac{c_\alpha}{L_{\mathcal D}}
$ with $
L_{\mathcal D}=\max_i\frac{\|a_i\|_2^2}{p_i},
$
and tune $c_\alpha$ over a logarithmically spaced grid. For AVRRK-AS, we first tune $(c_\alpha,\eta)$, followed by the momentum multipliers $(c_\theta,c_\gamma)$, and finally $\theta_2$. We employ successive halving with progressively increasing computational budgets to eliminate poorly performing configurations. Each selected configuration is further subjected to a long-horizon stability test. The resulting parameters are then fixed and used in all independent test trials.

During our  experiments, both VRRK-AS and AVRRK-AS employ uniform sampling with a mini-batch size of one. REK uses the standard row- and column-norm sampling schemes, whereas RCD uses column-norm sampling.
We measure accuracy by the relative solution error
$
\operatorname{RSE}
=
\frac{\|x^k-A^\dagger b\|_2^2}
{\|x^0-A^\dagger b\|_2^2}$. We set $\delta=10^{-1}$, $q=5$, and each experiment is repeated over $20$ independent trials.

\subsection{Comparison between VRRK-AS and AVRRK-AS}
\label{subsec:vrrk-avrrk-comparison}

In this subsection, we isolate the effect of acceleration on synthetic inconsistent systems with $m=2000$ and $n=r=200$.  Figures~\ref{fig:random-vrrk-avrrk} and \ref{fig:tail-vrrk-avrrk} show a clear transition with respect to the tested values of $\kappa$.  The bold line illustrates the median value over $20$ trials. The lightly shaded area signifies the  range from the minimum to the maximum values, while the darker shaded one indicates the data lying between the $25$-th and $75$-th quantiles.  At the smaller tested condition numbers,
VRRK-AS is faster in both epochs and CPU time. In this regime, its
variance-reduced iteration already contracts rapidly, and the reduction
in iteration count obtained from momentum is insufficient to offset the
additional recurrences  of AVRRK-AS.
Acceleration therefore does not improve practical performance uniformly.


\begin{figure}[htbp]
 \centering
 \begin{tabular}{@{}ccc@{}}
  \includegraphics[width=0.32\linewidth]{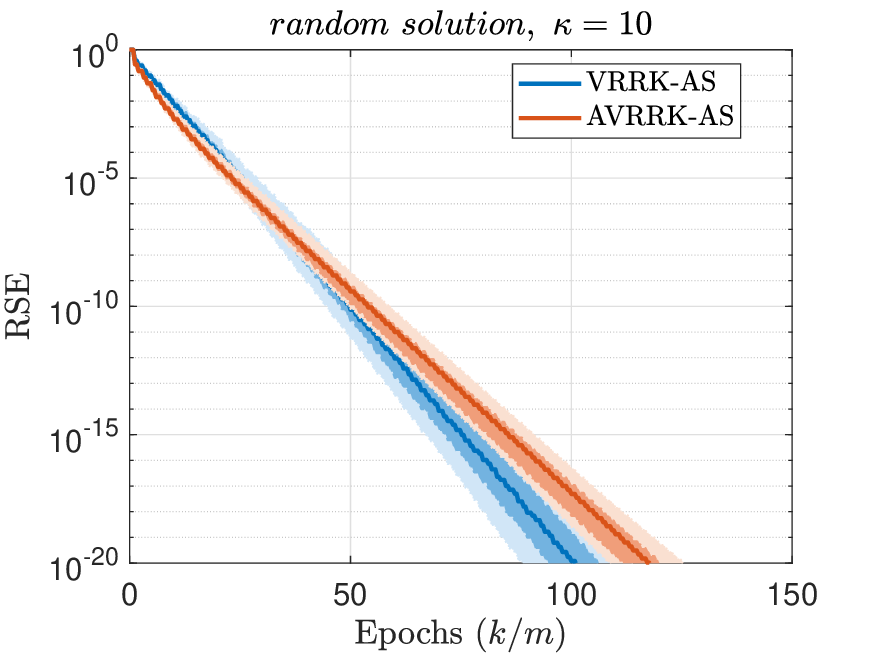}&
  \includegraphics[width=0.32\linewidth]{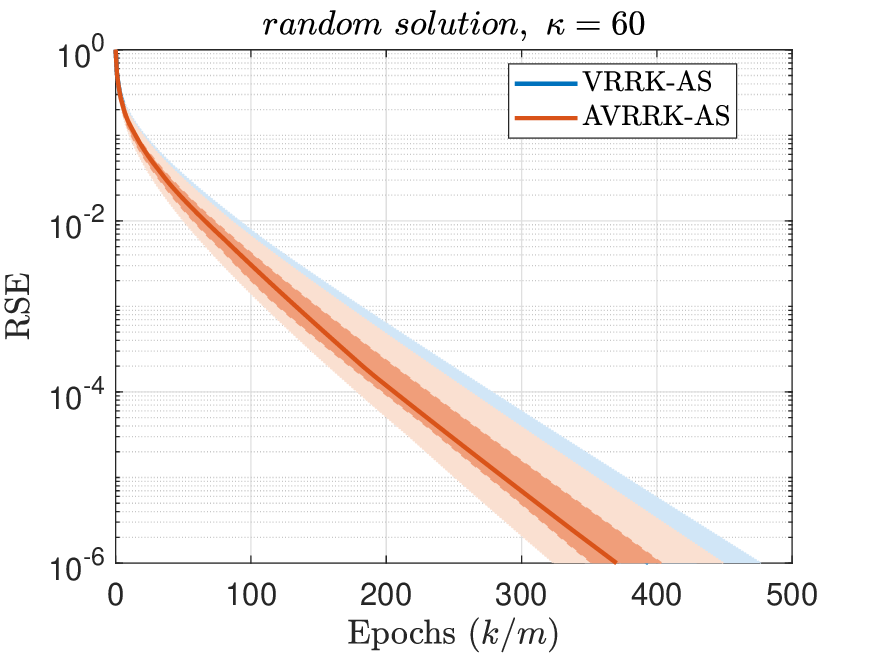}&
  \includegraphics[width=0.32\linewidth]{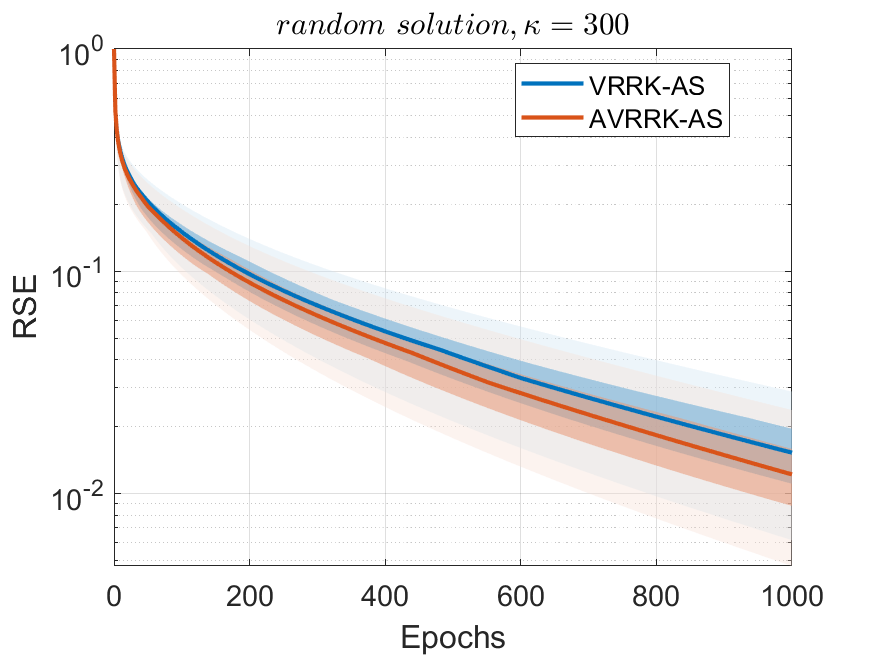}\\
  \includegraphics[width=0.32\linewidth]{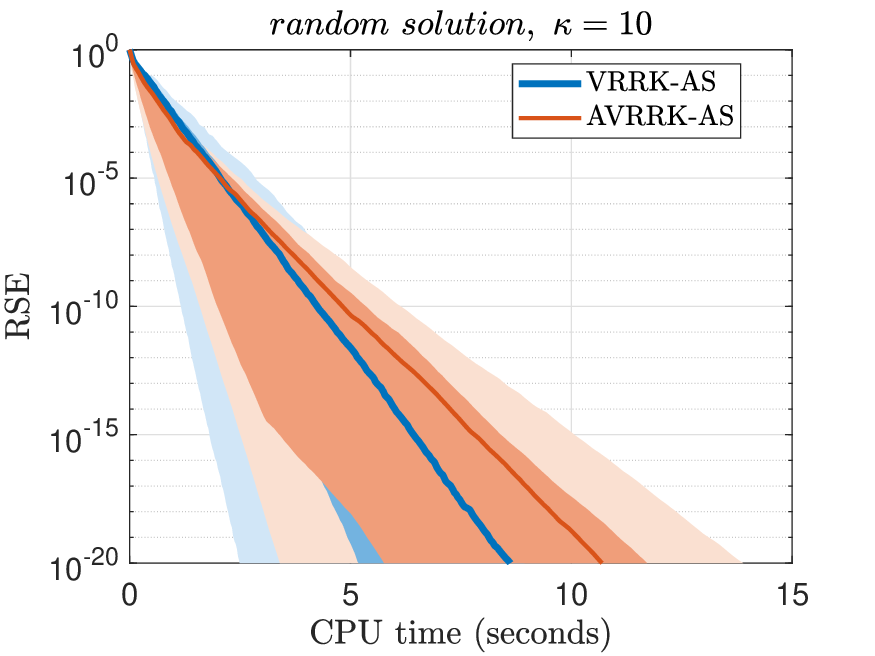}&
  \includegraphics[width=0.32\linewidth]{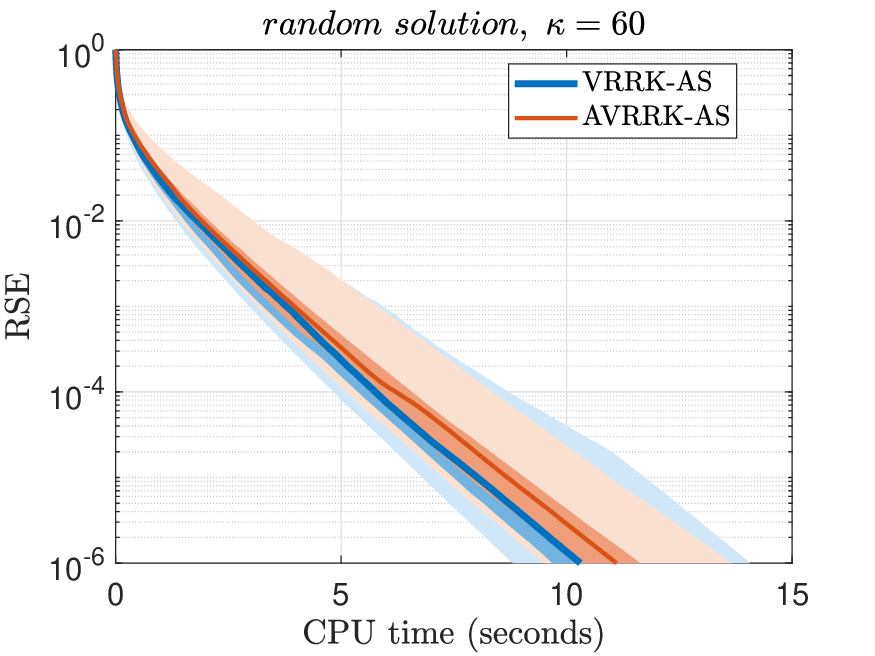}&
  \includegraphics[width=0.32\linewidth]{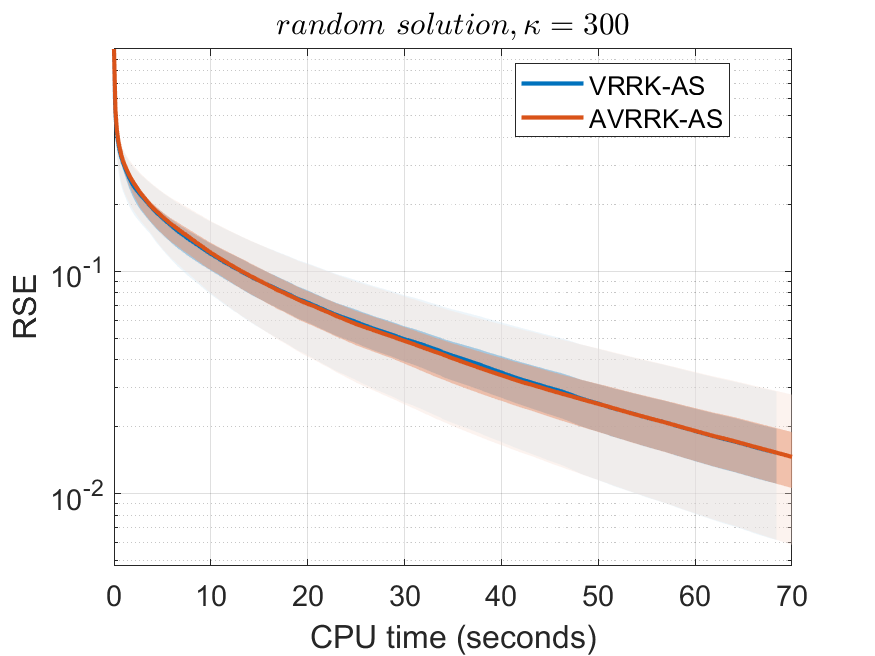}
 \end{tabular}
 \caption{Comparison of VRRK-AS and AVRRK-AS for random solutions with
 $\kappa=10$, $60$, and $300$ (from left to right). The top row shows RSE
 versus epochs, and the bottom row shows RSE versus CPU time.}
 \label{fig:random-vrrk-avrrk}
\end{figure}

\begin{figure}[htbp]
 \centering
 \begin{tabular}{@{}ccc@{}}
  \includegraphics[width=0.32\linewidth]{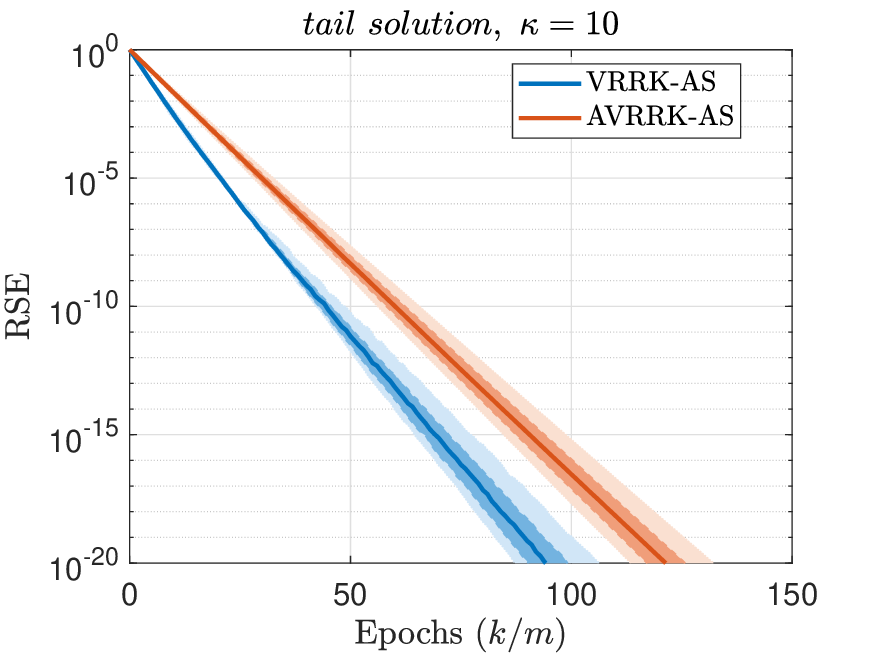}&
  \includegraphics[width=0.32\linewidth]{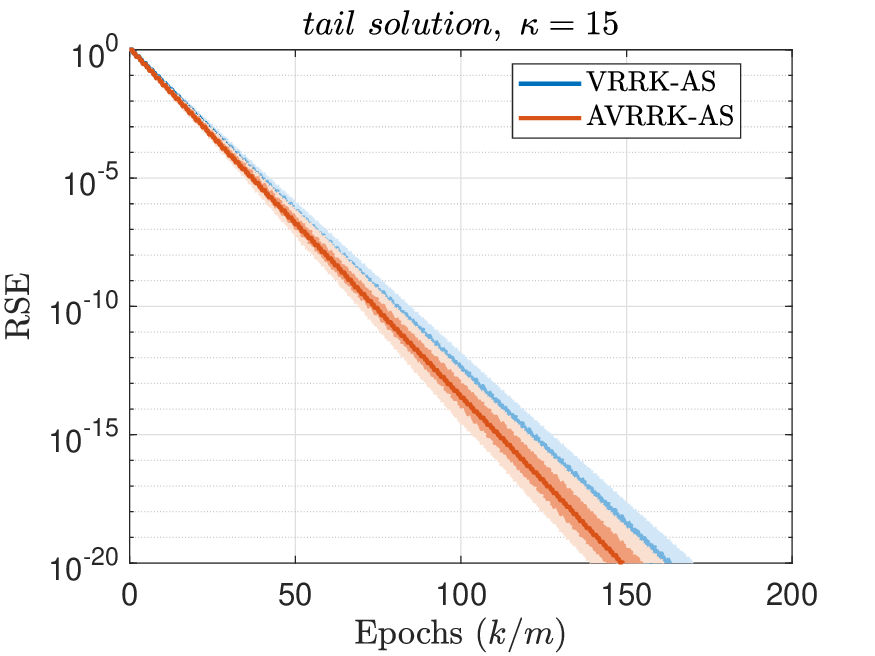}&
  \includegraphics[width=0.32\linewidth]{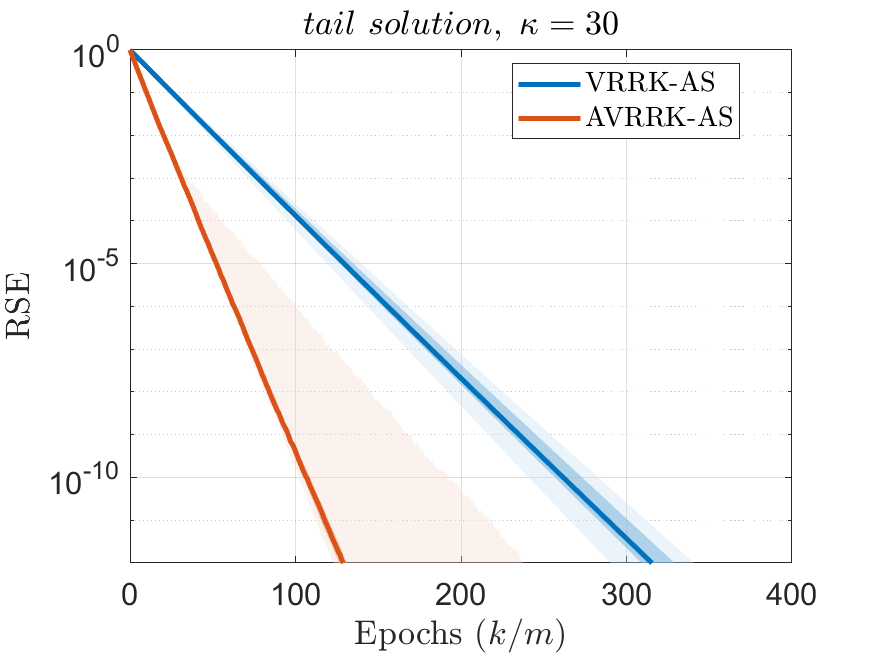}\\
  \includegraphics[width=0.32\linewidth]{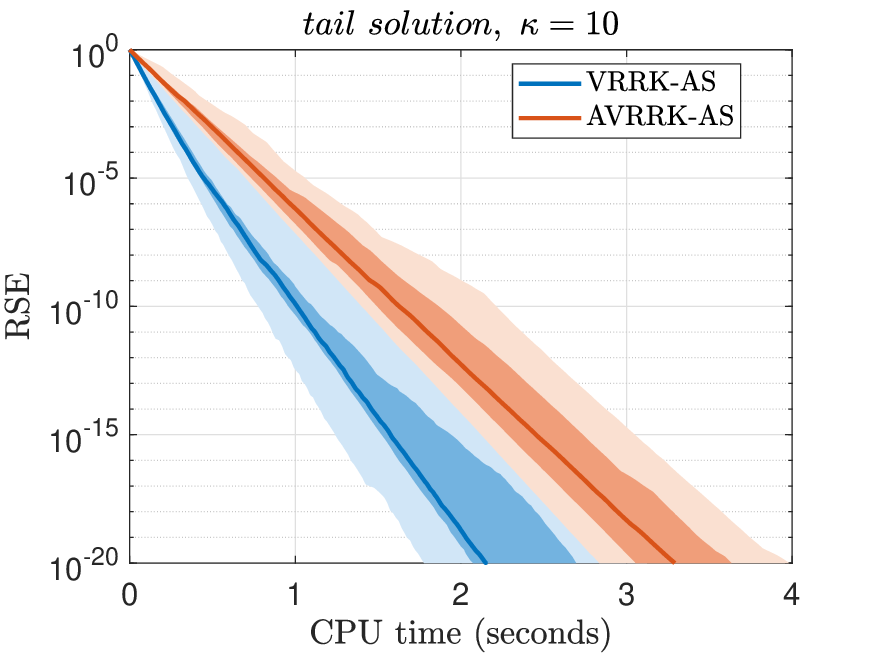}&
  \includegraphics[width=0.32\linewidth]{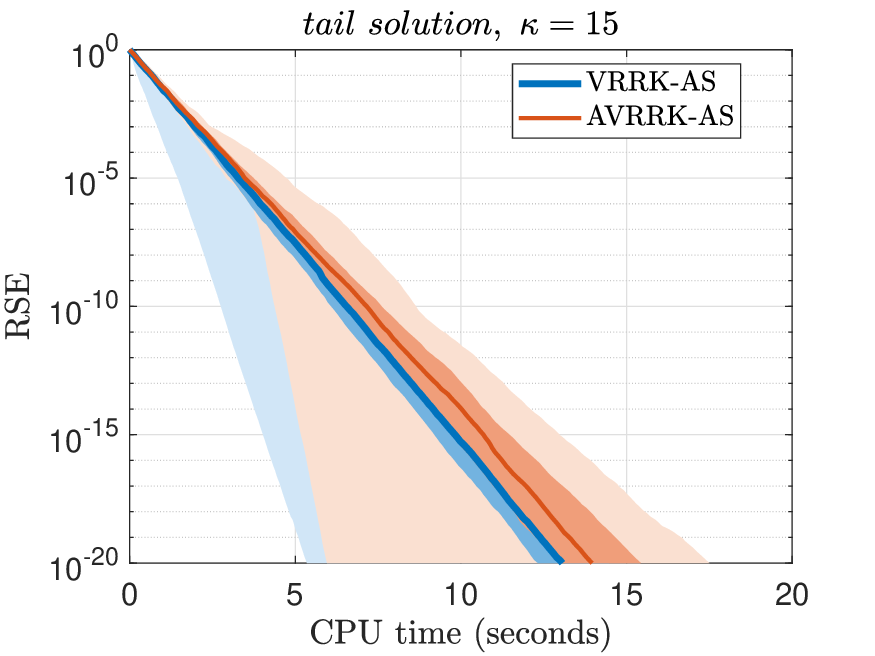}&
  \includegraphics[width=0.32\linewidth]{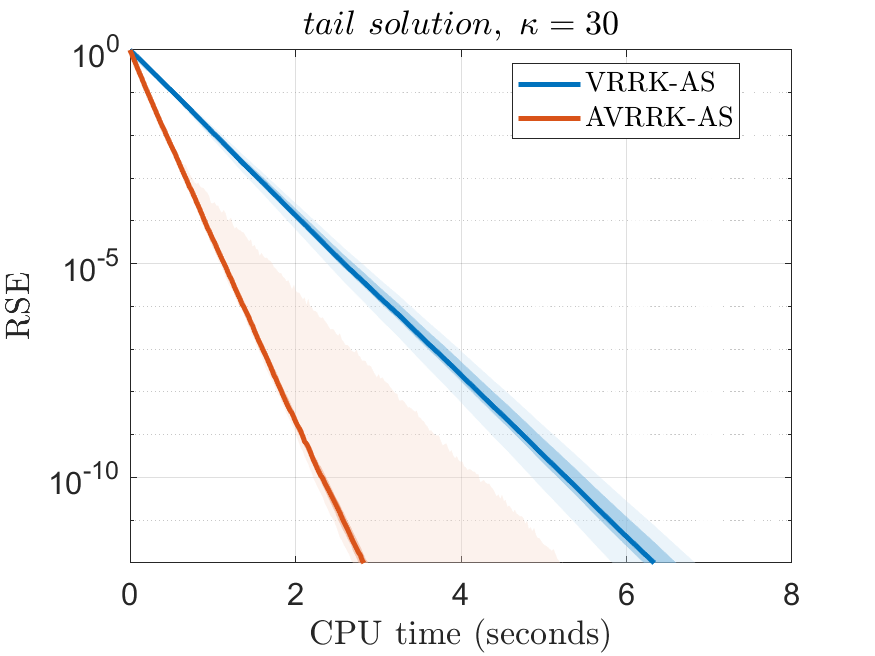}
 \end{tabular}
 \caption{Comparison of VRRK-AS and AVRRK-AS for tail-aligned solutions
 with $\kappa=10$, $15$, and $30$ (from left to right). The top row shows
 RSE versus epochs, and the bottom row shows RSE versus CPU time.}
 \label{fig:tail-vrrk-avrrk}
\end{figure}

As $\kappa$ increases, the convergence of VRRK-AS is increasingly limited
by components associated with the small positive singular values. The
advantage then shifts toward AVRRK-AS: its median RSE decreases more
rapidly with epochs and, for the larger tested values of $\kappa$, the
iteration reduction also yields a shorter CPU time. The transition is
particularly visible for the tail-aligned solution because the initial
error is concentrated in the slow spectral subspace. This behavior is
consistent with the accelerated term in Theorem~\ref{thm-acc} and
Remark~\ref{remark-quadratic}, whose dependence involves
$\sqrt{\alpha\sigma_{\min}^2(A)}$, in contrast to the corresponding
$\alpha\sigma_{\min}^2(A)$ dependence for VRRK-AS. 

\subsection{Comparison with REK and RCD}
\label{subsec:comparison-existing}

We  compare VRRK-AS and AVRRK-AS with REK and RCD. All systems remain
inconsistent. RCD is included only for full-column-rank matrices, because
for a rank-deficient matrix its iterates are not generally guaranteed to
converge to   $A^\dagger b$. REK and
the two proposed methods are used in both the full-rank and rank-deficient
experiments.

For the full-column-rank synthetic systems, we set $m=2000$ and $n=r=200$.
Figures~\ref{fig:tail-vrrk-avrrk-rek-rcd} and
\ref{fig:random-vrrk-avrrk-rek-rcd} compare all four methods for
tail-aligned and random solutions.  It can be seen that at smaller $\kappa$, RCD is highly competitive and
often requires the least CPU time even though it needs more epochs
than the proposed methods, reflecting the low cost of a coordinate update.
As $\kappa$ increases, REK and RCD slow substantially, while variance
reduction gives VRRK-AS and AVRRK-AS a pronounced advantage in epoch
count. In the most ill-conditioned cases, AVRRK-AS also has the best
CPU time performance. The transition is especially clear for the
tail-aligned solution at $\kappa=30$. For the rank-deficient synthetic systems, we  fix $m=2000$, $n=200$, $\kappa=100$, and
consider $r\in\{50,100,150\}$.  As shown in
Figure~\ref{fig:random-vrrk-avrrk-rek-rcd-def}, both proposed methods
reduce the RSE more rapidly than REK over most of the tested accuracy
range, and AVRRK-AS requires the least CPU time to attain common RSE
levels in all three cases.

\begin{figure}[htbp]
 \centering
 \begin{tabular}{@{}ccc@{}}
  \includegraphics[width=0.32\linewidth]{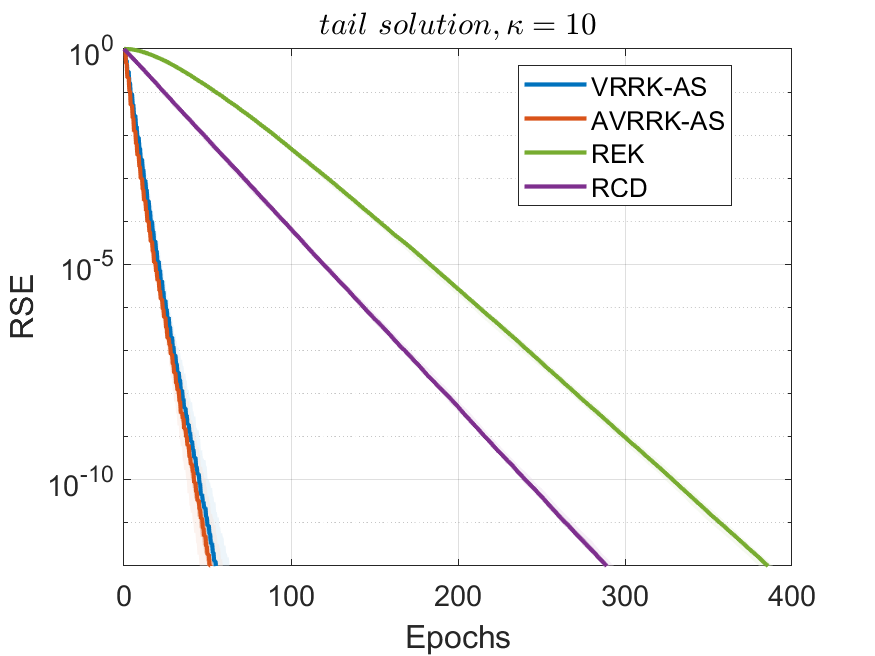}&
  \includegraphics[width=0.32\linewidth]{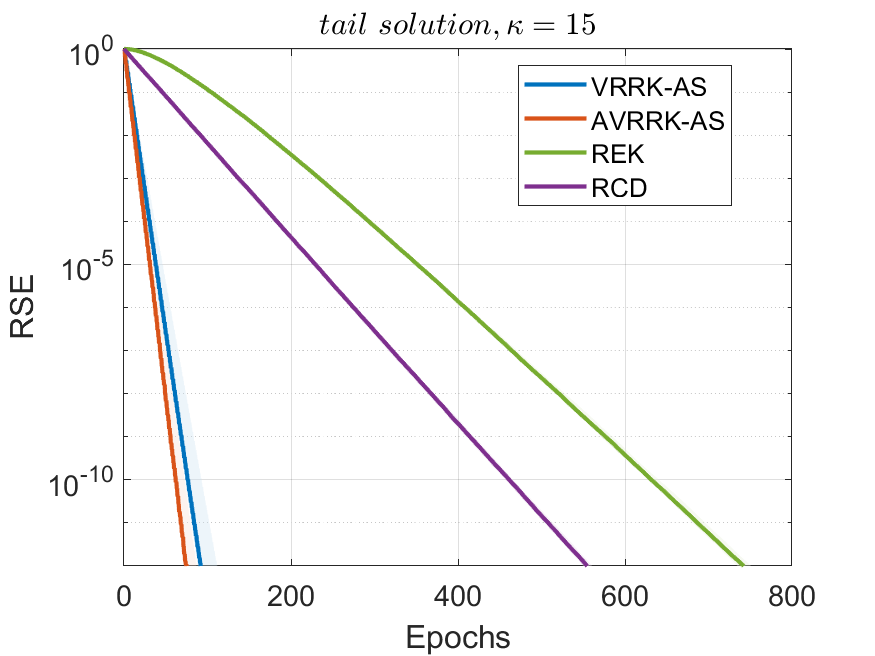}&
  \includegraphics[width=0.32\linewidth]{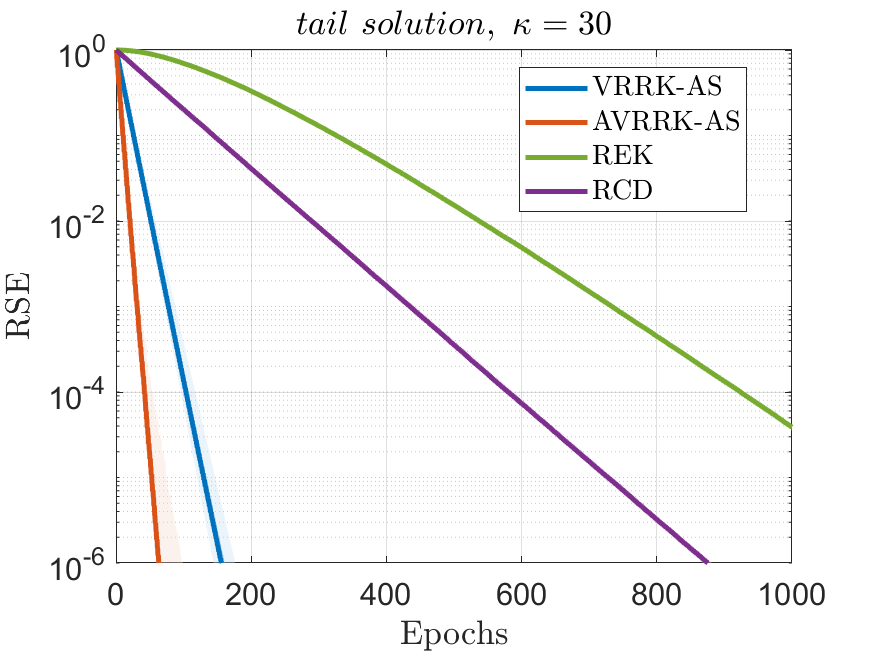}\\
  \includegraphics[width=0.32\linewidth]{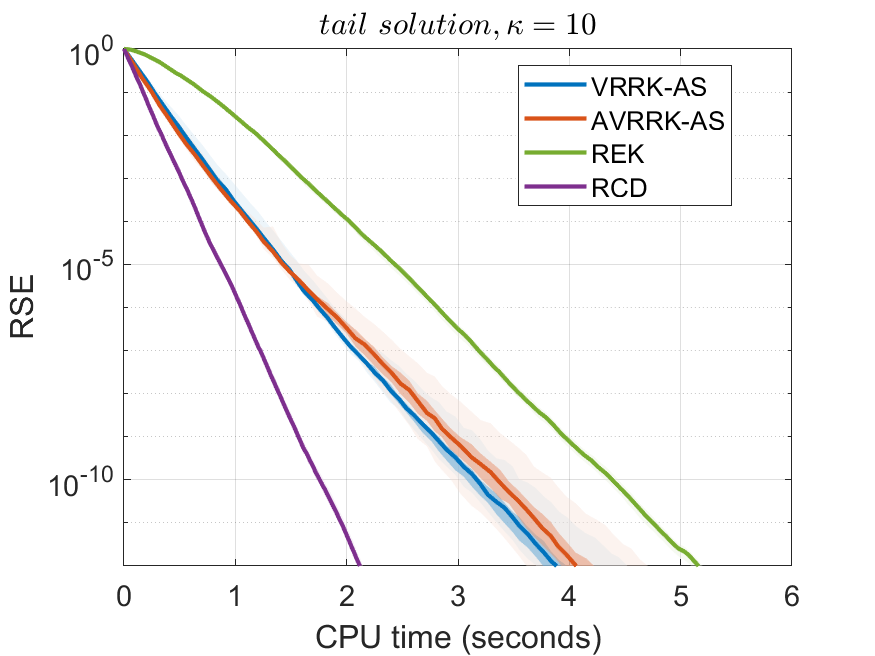}&
  \includegraphics[width=0.32\linewidth]{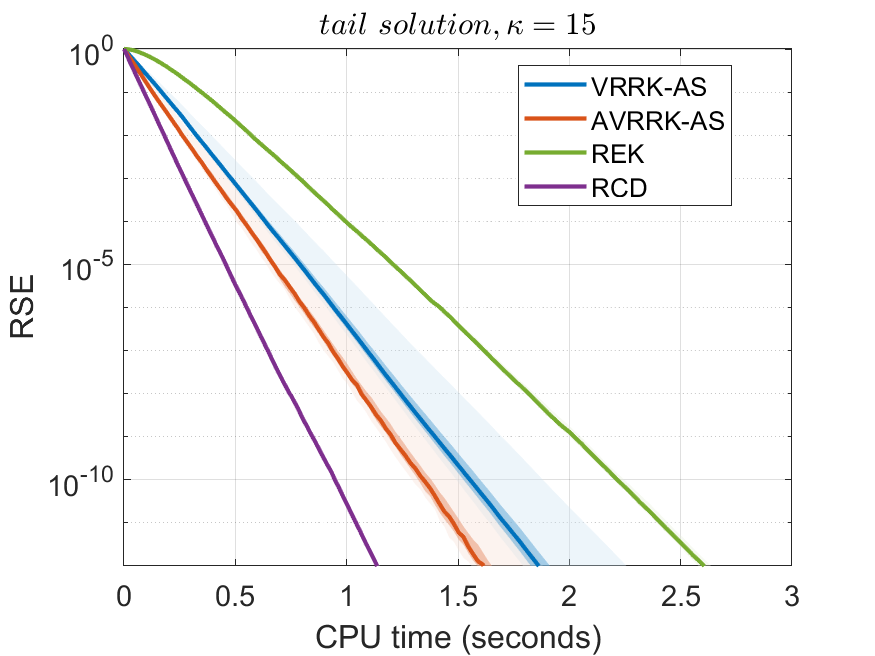}&
  \includegraphics[width=0.32\linewidth]{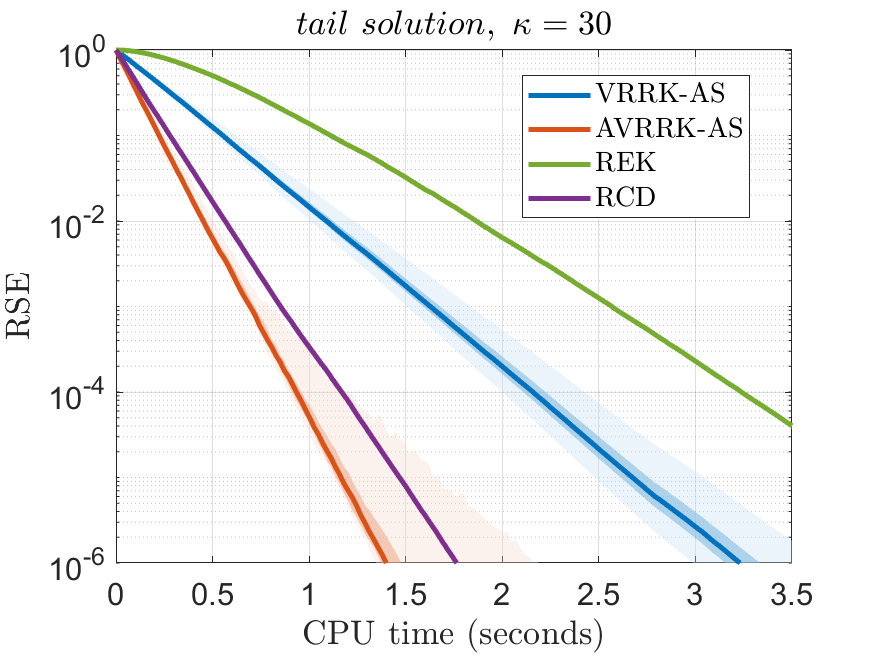}
 \end{tabular}
 \caption{Comparison of VRRK-AS, AVRRK-AS, REK, and RCD on
 full-column-rank synthetic systems with tail-aligned solutions. From left
 to right, $\kappa=10$, $15$, and $30$. The top and bottom rows report
 epochs and CPU time, respectively.}
 \label{fig:tail-vrrk-avrrk-rek-rcd}
\end{figure}

\begin{figure}[htbp]
 \centering
 \begin{tabular}{@{}ccc@{}}
  \includegraphics[width=0.32\linewidth]{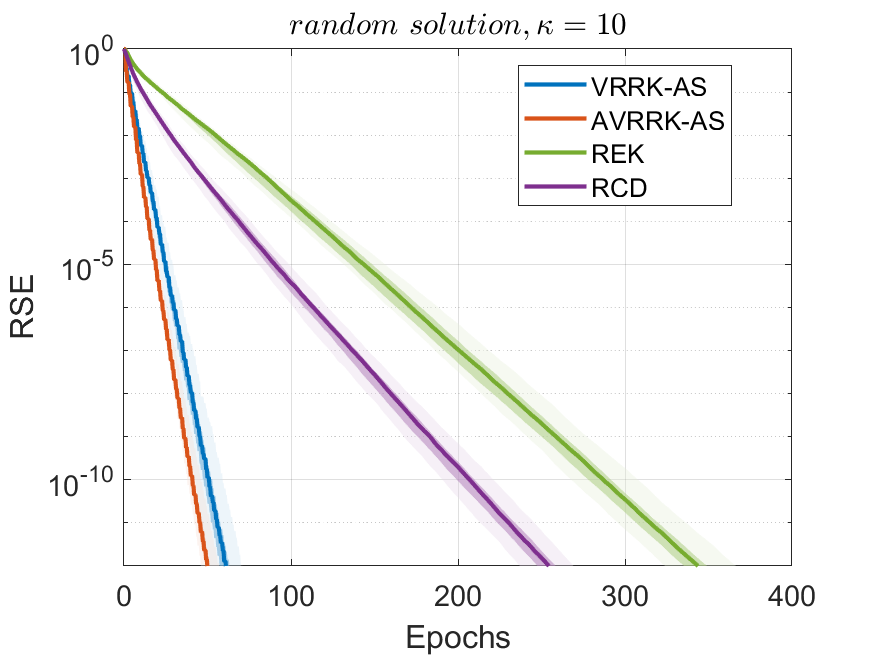}&
  \includegraphics[width=0.32\linewidth]{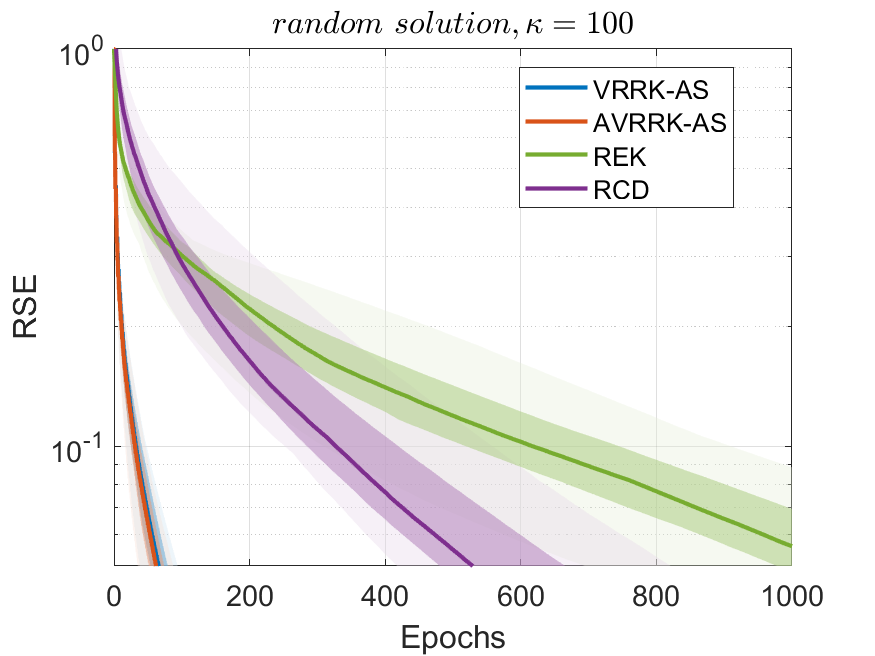}&
  \includegraphics[width=0.32\linewidth]{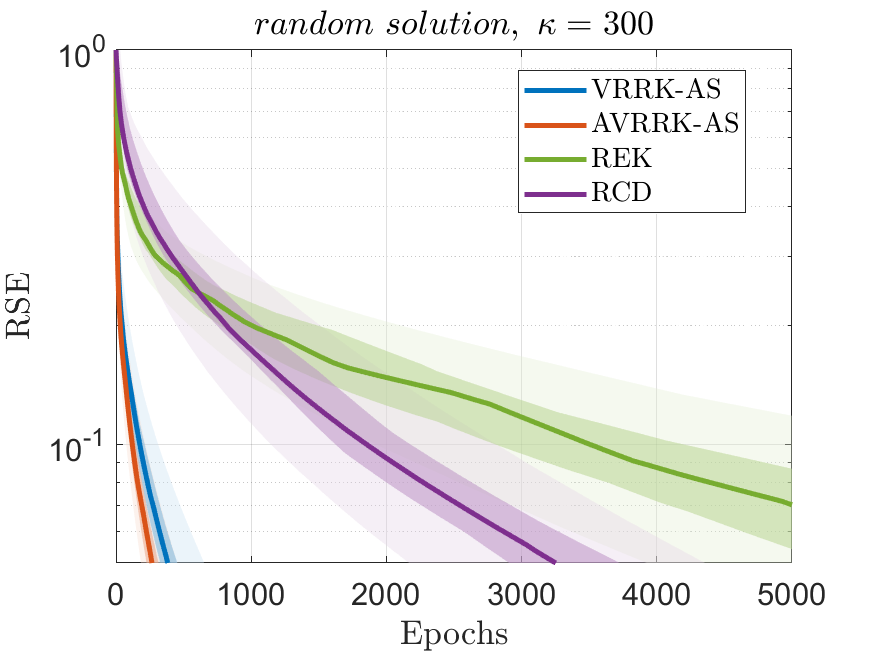}\\
  \includegraphics[width=0.32\linewidth]{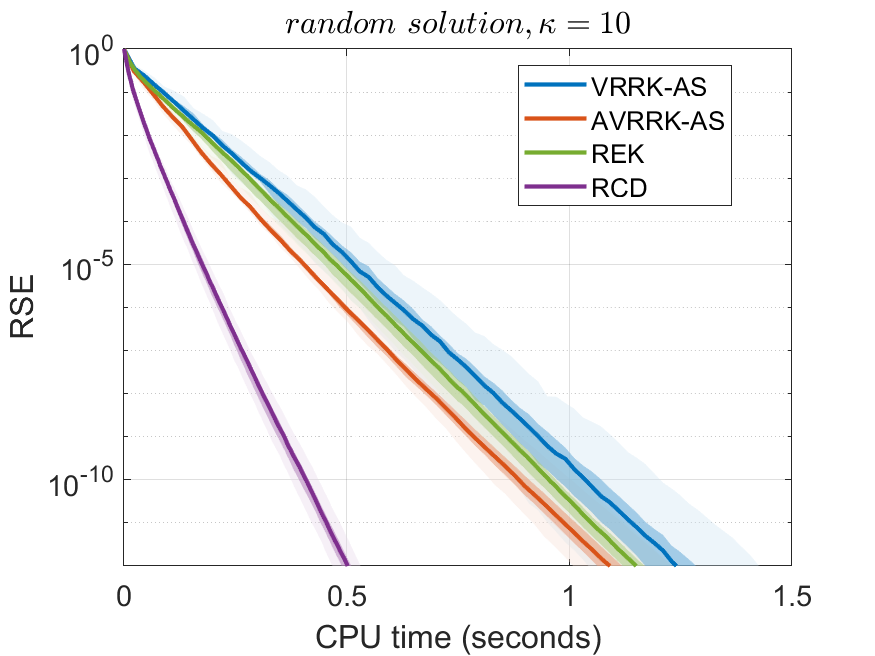}&
  \includegraphics[width=0.32\linewidth]{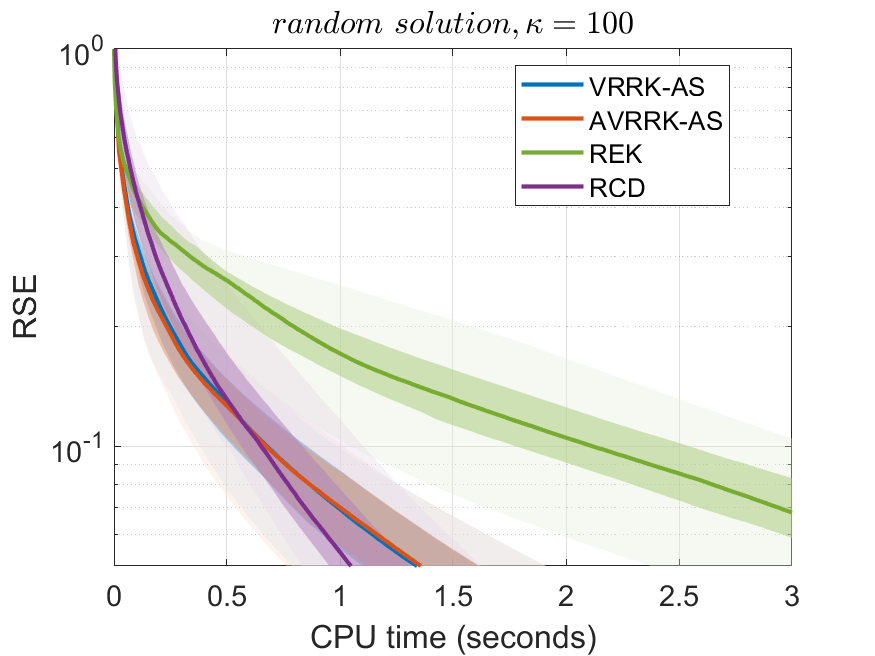}&
  \includegraphics[width=0.32\linewidth]{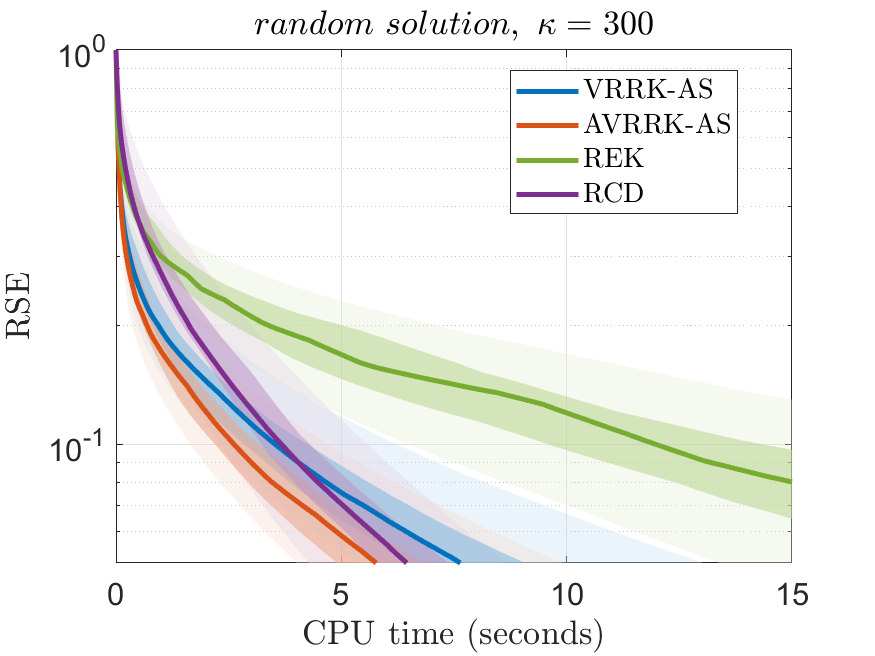}
 \end{tabular}
 \caption{Comparison of VRRK-AS, AVRRK-AS, REK, and RCD on
 full-column-rank synthetic systems with random solutions. From left to
 right, $\kappa=10$, $100$, and $300$. The top and bottom rows report
 epochs and CPU time, respectively.}
 \label{fig:random-vrrk-avrrk-rek-rcd}
\end{figure}


\begin{figure}[htbp]
 \centering
 \begin{tabular}{@{}ccc@{}}
  \includegraphics[width=0.32\linewidth]{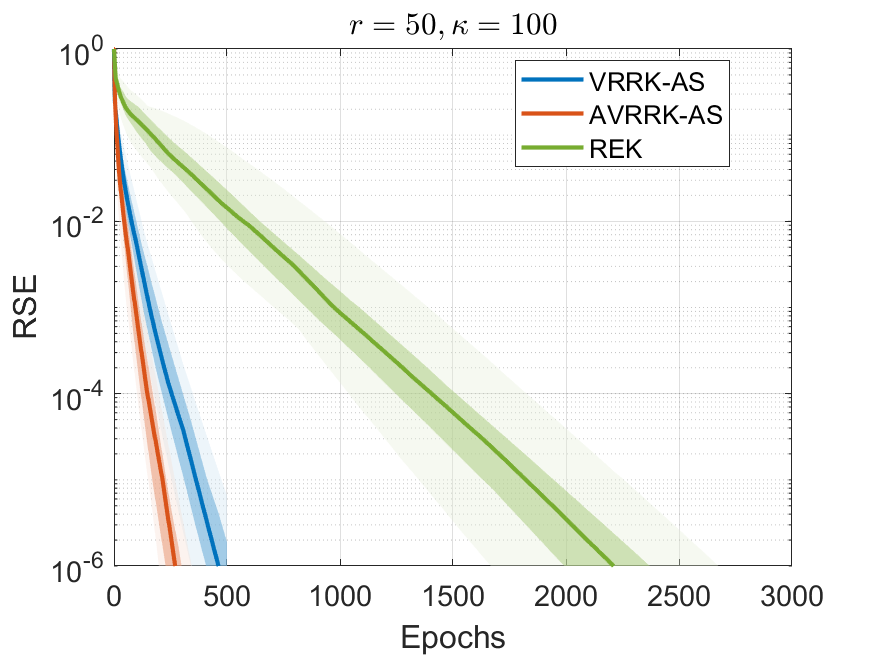}&
  \includegraphics[width=0.32\linewidth]{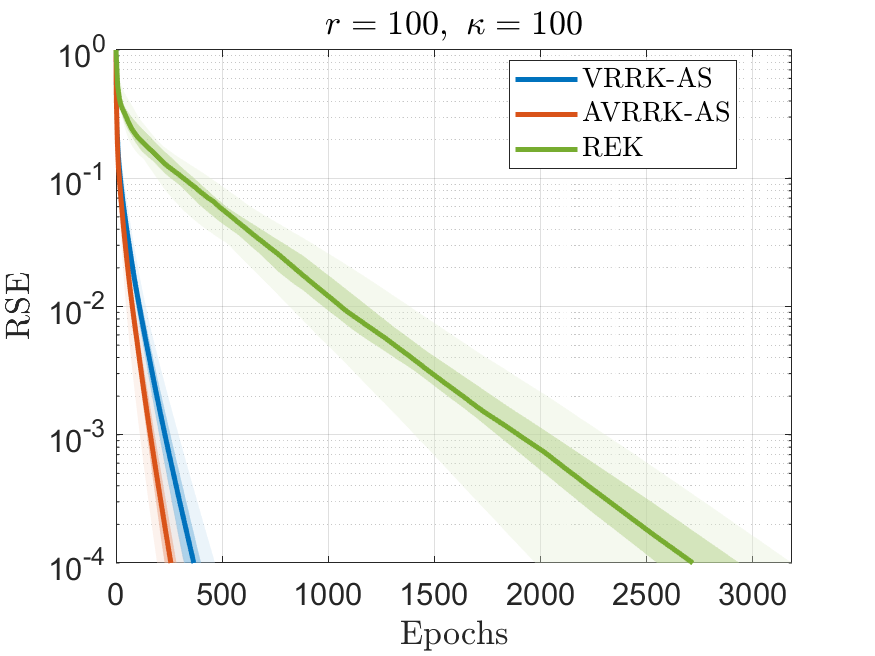}&
  \includegraphics[width=0.32\linewidth]{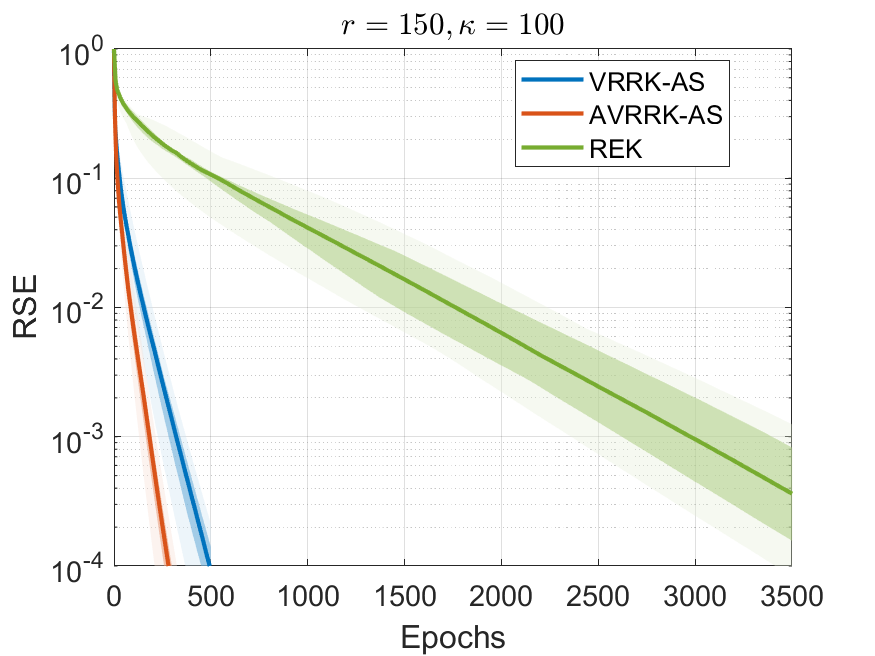}\\
  \includegraphics[width=0.32\linewidth]{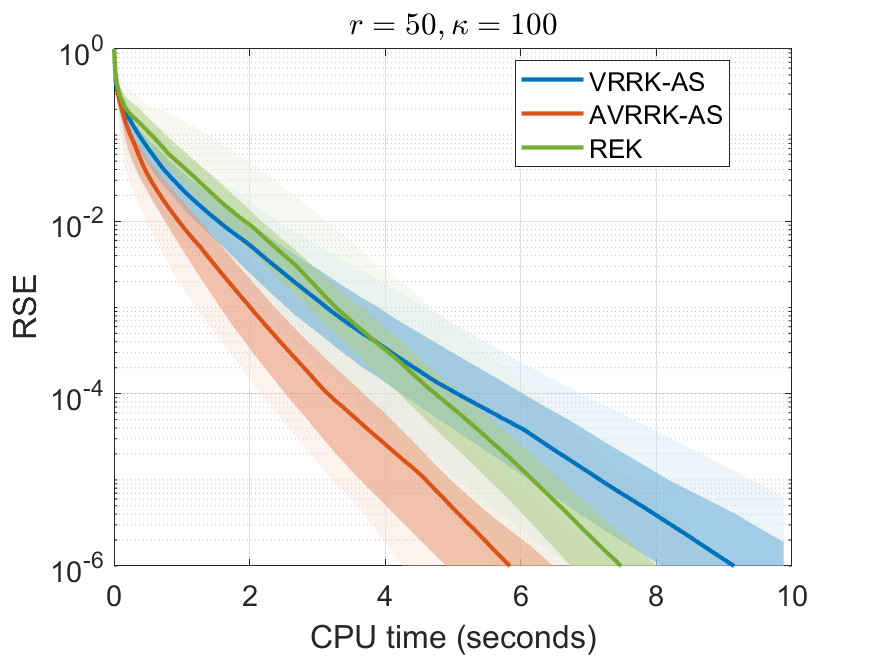}&
  \includegraphics[width=0.32\linewidth]{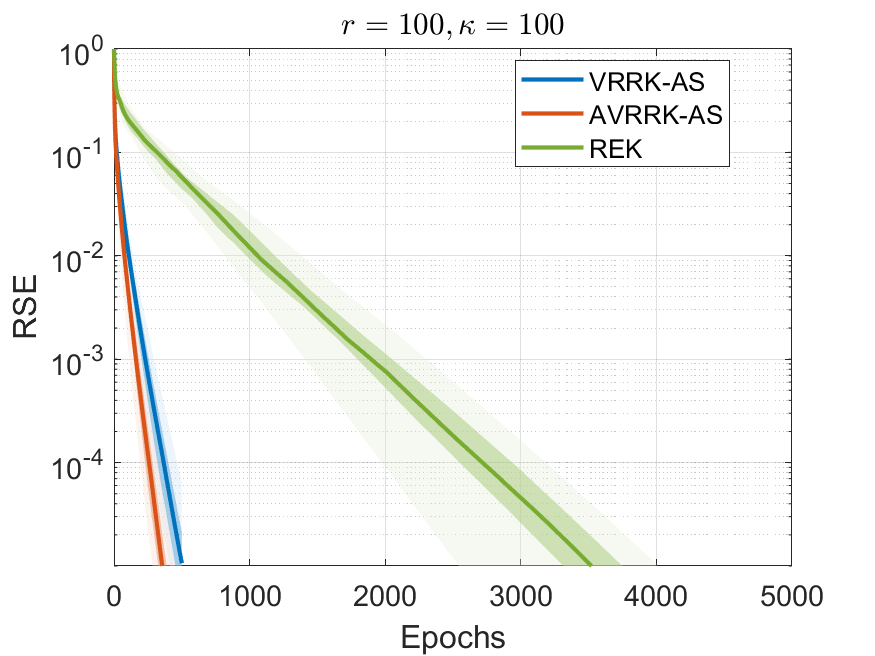}&
  \includegraphics[width=0.32\linewidth]{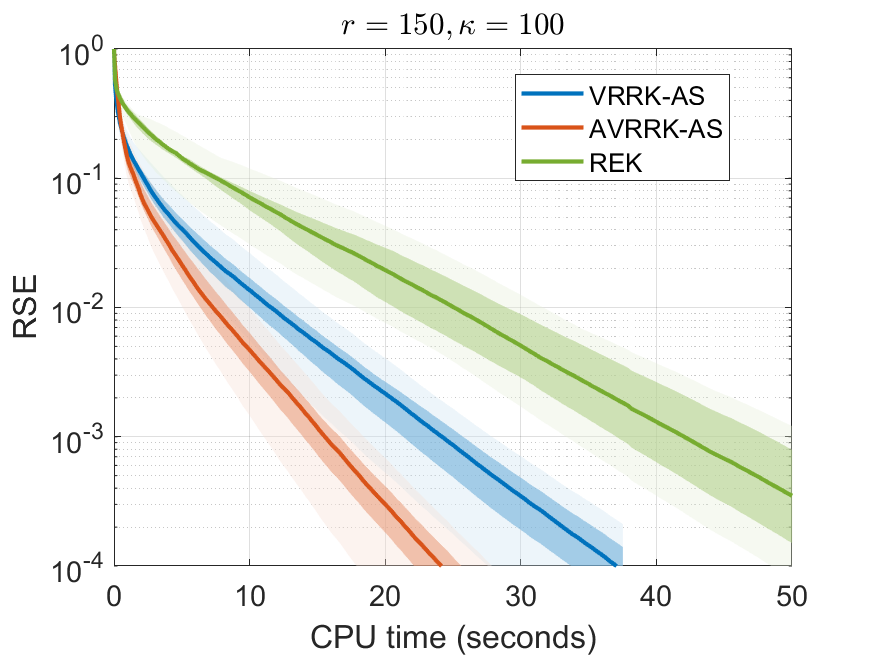}
 \end{tabular}
 \caption{Comparison of VRRK-AS, AVRRK-AS, and REK on rank-deficient
 synthetic systems with random solutions and $\kappa=100$. From left to
 right, $r=50$, $100$, and $150$. RCD is omitted because these matrices
 are not of full column rank. The top and bottom rows report epochs and
 CPU time, respectively.}
 \label{fig:random-vrrk-avrrk-rek-rcd-def}
\end{figure}
 
Finally, we compare all four methods on the full-column-rank matrices
\texttt{ash958}, \texttt{well1033}, and \texttt{illc1033}. These problems
complement the controlled synthetic tests with sparse and irregular row
and column structures. Figure~\ref{fig:sparse-vrrk-avrrk-rek-rcd} shows
that no method dominates on every matrix. On \texttt{ash958}, all methods
reduce the RSE rapidly, and RCD has a clear CPU time advantage. On
\texttt{well1033}, AVRRK-AS reaches the lowest median RSE and improves on
the other methods in late-stage accuracy. On the difficult
\texttt{illc1033} problem, none of the methods reaches high accuracy within
the prescribed budget; AVRRK-AS nevertheless attains the smallest median
terminal RSE, whereas RCD terminates at a substantially larger error.

\begin{figure}[htbp]
 \centering
 \begin{tabular}{@{}ccc@{}}
  \includegraphics[width=0.32\linewidth]{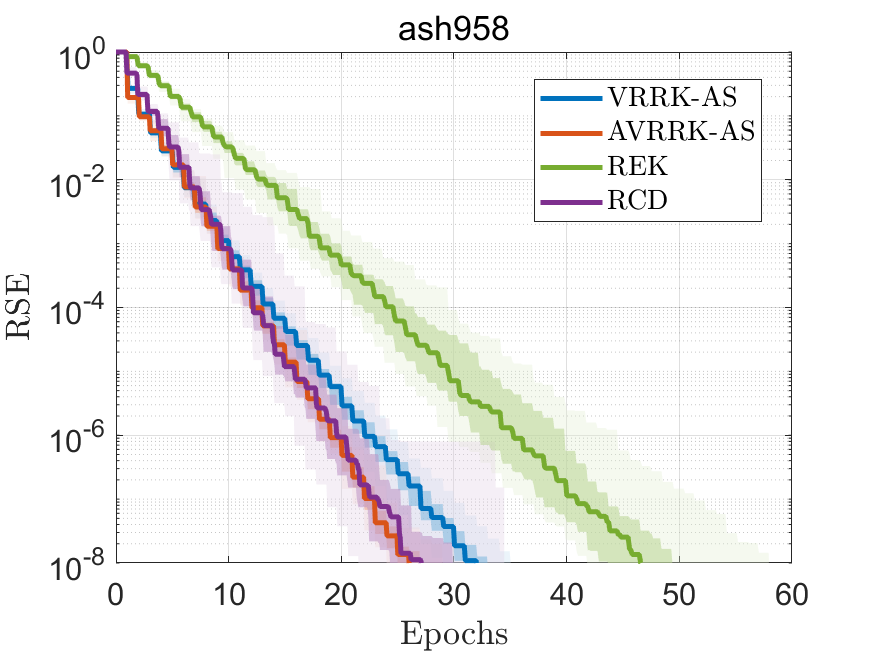}&
  \includegraphics[width=0.32\linewidth]{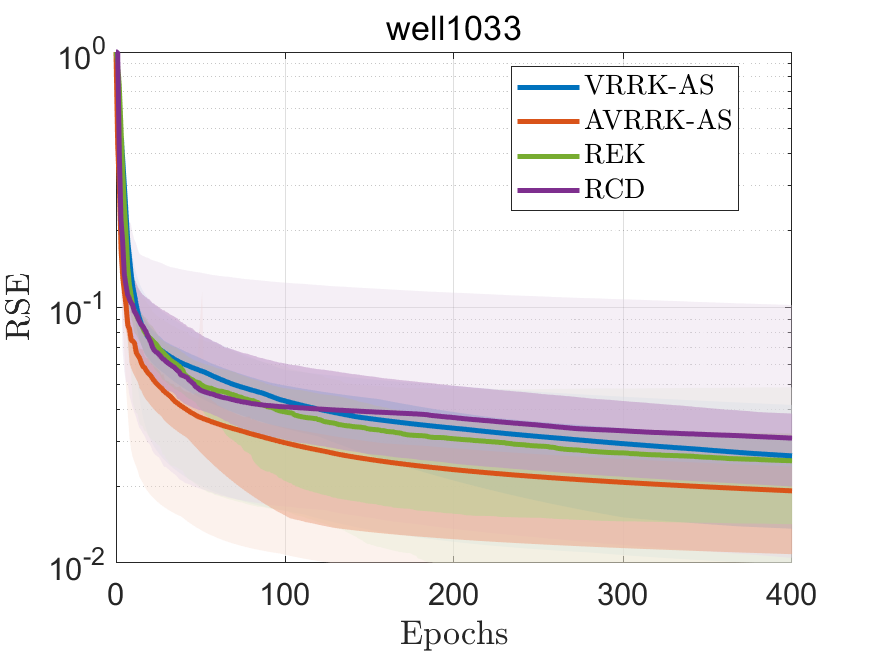}&
  \includegraphics[width=0.32\linewidth]{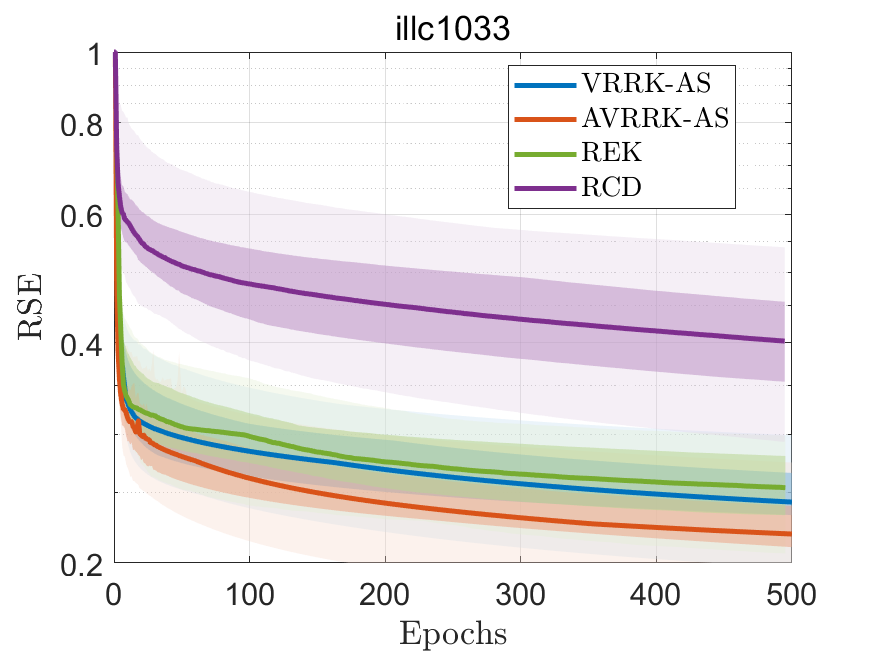}\\
  \includegraphics[width=0.32\linewidth]{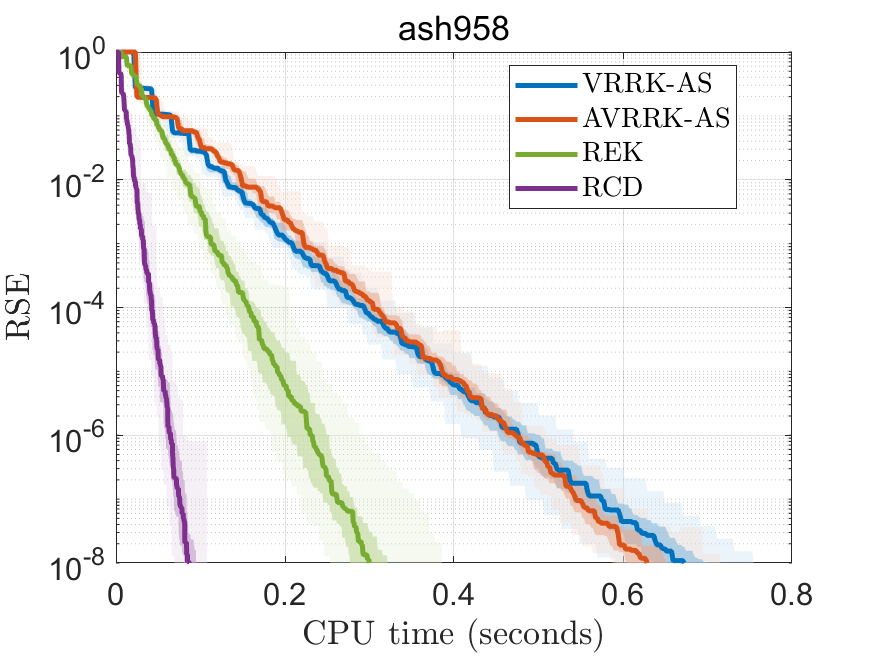}&
  \includegraphics[width=0.32\linewidth]{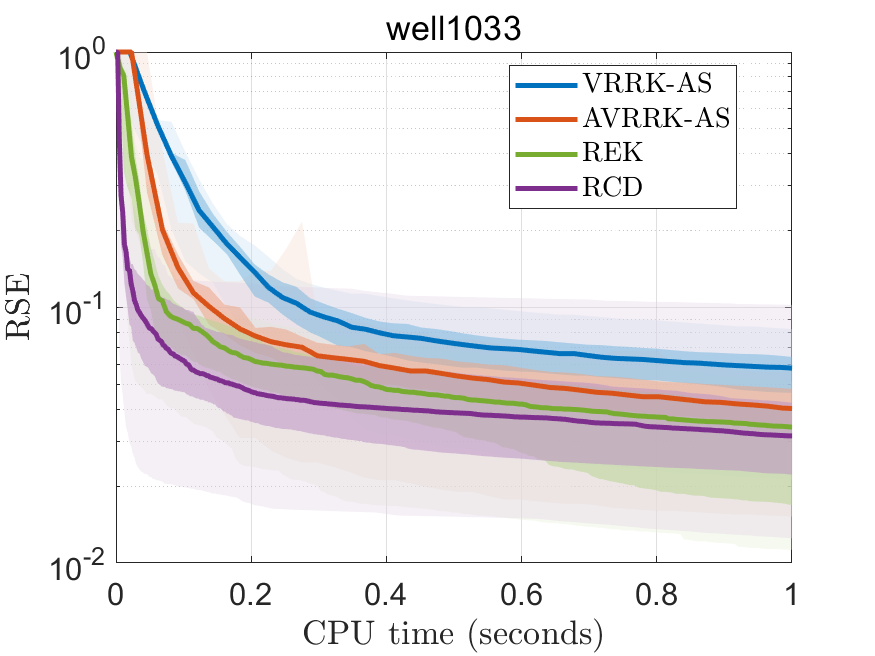}&
  \includegraphics[width=0.32\linewidth]{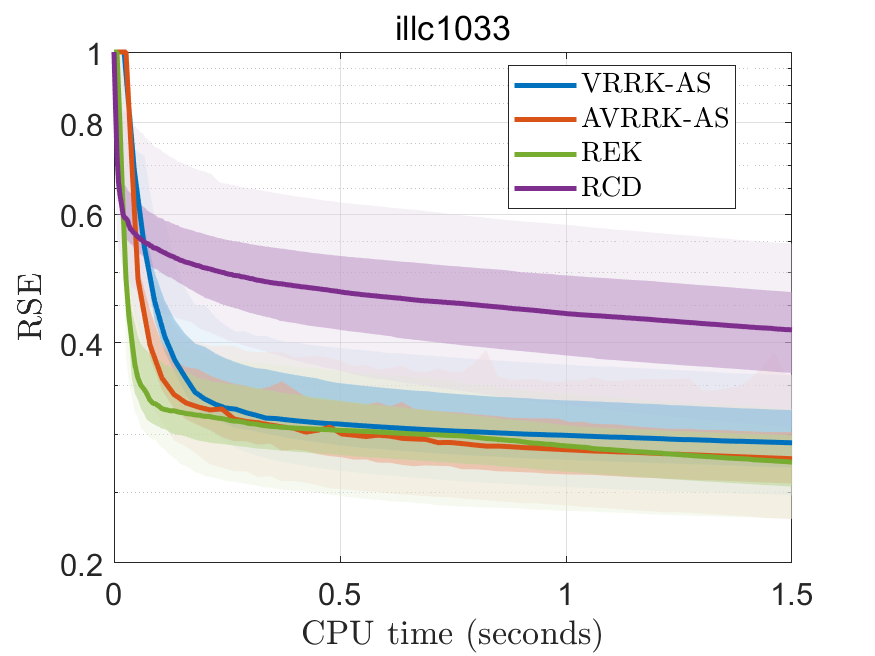}
 \end{tabular}
 \caption{Comparison of VRRK-AS, AVRRK-AS, REK, and RCD on three
 full-column-rank SuiteSparse matrices. From left to right,
 \texttt{ash958}, \texttt{well1033}, and \texttt{illc1033}. The top and
 bottom rows report epochs and CPU time, respectively.}
 \label{fig:sparse-vrrk-avrrk-rek-rcd}
\end{figure}

\section{Concluding remarks}

 We presented a residual-tracking formulation of variance-reduced
 randomized Kaczmarz methods and established its connection with
 JacSketch for least-squares problems.
 The formulation admits a row-only implementation with
 \(O(m+n)\) auxiliary storage, and our analysis establishes linear
 convergence to the minimum-norm least-squares solution, including
 for rank-deficient systems.
 We also developed and analyzed a momentum-accelerated extension.
 Numerical experiments show that VRRK-AS is generally more economical for the smaller tested values of condition number, whereas
AVRRK-AS becomes more effective as the problem becomes increasingly
ill-conditioned.

There are still many possible future avenues of research.
The performance of AVRRK-AS depends on multiple interacting parameters, and developing
adaptive or theoretically guided parameter-selection strategies would
improve its practical robustness \cite{zeng2023adaptive,sun2025connecting}. It would also be useful to investigate
nonuniform and adaptive sampling rules, block and mini-batch variants,
and parallel implementations for large-scale sparse problems \cite{xiao2026rplss,brust2023plss,zeng2026randomized}. 

\bibliography{main1230}

\section{Appendix. Proof of the main results}

For convenience, we write $x^* = A^\dagger b$. The following lemma is essential for our proofs.
\begin{lemma}\label{error-bound-lemma}
	Let $A \in \mathbb{R}^{m\times n}$ be any nonzero real matrix. For every $u \in  \operatorname{Range}(A^\top )$,
	it holds that
	$$
	\|A  u\|^2_2\geq \sigma_{\min}^2(A)\|u\|^2_2.
	$$
\end{lemma}

	\subsection{Proof of Theorem \ref{prop-rk-star} and  Proposition \ref{THM-EQ}}

\begin{proof}[Proof of Theorem \ref{prop-rk-star}]
	From the update rule \eqref{RK-star}, we have
	$$
	\begin{aligned}
		&\|x^{k+1}-x^*\|^2_2=\left\| x^k-x^*-\frac{ \langle a_{i_k}, x^k\rangle - b_{i_k}-r^*_{i_k}}{\|a_{i_k}\|_2^2} a_{i_k} \right\|^2_2\\
		&=\|x^k-x^*\|^2_2-2\left\langle x^k-x^*,\frac{\langle a_{i_k}, x^k\rangle - b_{i_k}-r^*_{i_k}}{\|a_{i_k}\|_2^2} a_{i_k} \right\rangle+\frac{(\langle a_{i_k}, x^k\rangle - b_{i_k}-r^*_{i_k})^2}{\|a_{i_k}\|_2^2}. 
	\end{aligned}
	$$
	Hence
	$$
	\begin{aligned}
		\mathbb{E}[\|x^{k+1}-x^*\|^2_2\mid x^k]=&\|x^k-x^*\|^2_2-\frac{2}{\|A\|^2_F}\left\langle x^k-x^*,A^\top(Ax^k-b-r^*) \right\rangle\\
		&+\frac{\|Ax^k-b-r^*\|^2_2}{\|A\|^2_F}\\
		=&\|x^k-x^*\|^2_2-\frac{\|A(x^k-x^*)\|^2_2}{\|A\|^2_F}.
	\end{aligned}
	$$
	Since $x^0 \in \operatorname{Range}(A^\top)$ and the update \eqref{RK-star} preserves this subspace (each step adds a vector in $\operatorname{Range}(A^\top)$), we have $x^k \in \operatorname{Range}(A^\top)$ for all $k$. Consequently, $x^k - x^* \in \operatorname{Range}(A^\top)$. 
	Lemma \ref{error-bound-lemma} then implies
	$
	\|A(x^k - x^*)\|_2^2 \ge \sigma_{\min}^2(A) \|x^k - x^*\|_2^2.
	$
	Substituting this bound into the conditional expectation yields
	\[
	\mathbb{E} [\|x^{k+1} - x^*\|_2^2 \mid x^k ] 
	\le \left(1 - \frac{\sigma_{\min}^2(A)}{\|A\|_F^2}\right) \|x^k - x^*\|_2^2.
	\]
	Taking full expectation and unrolling the recursion completes the proof.
\end{proof}

\begin{proof}[Proof of Proposition \ref{THM-EQ}]
	It suffices to prove $g^k = \tilde{g}^k$ for all $k \ge 0$.
	Define $D_k := \operatorname{diag}(r(x^k))P^{-1}$ and $\Lambda_k := \operatorname{diag}(\bar{r}^k)P^{-1}$ with $P := \operatorname{diag}(p_1,\ldots,p_m)$. 
	Then $ {G}(x^k)=A^\top D_k$, $ {J}^0 = A^\top D_0$, and from \eqref{barr-update} that
	\[
	\Lambda_{k+1} = \Lambda_k + (D_k - \Lambda_k) \sum_{i\in\mathcal{I}_k} e_i e_i^\top.
	\]
	We claim $ {J}^k = A^\top \Lambda_k$ for all $k$. 
	For $k=0$, $ {J}^0 = A^\top D_0   = A^\top \Lambda_0$ since $\Lambda_0 = D_0  $. 
	Assuming the claim holds for $k$, then
	\[
	A^\top \Lambda_{k+1} = A^\top \Lambda_k + (A^\top D_k - A^\top \Lambda_k) \sum_{i\in\mathcal{I}_k} e_i e_i^\top
	=  {J}^k + ( {G}(x^k) -  {J}^k) \sum_{i\in\mathcal{I}_k} e_i e_i^\top =  {J}^{k+1}.
	\]
	Hence $ {J}^k p = A^\top \Lambda_k p = A^\top \bar{r}^k =e^k$ and
	\[
	\tilde{g}^k =  {J}^k p + ( {G}(x^k) -  {J}^k) \sum_{i\in\mathcal{I}_k} e_i
	= e^k + \sum_{i\in\mathcal{I}_k} \frac{1}{p_i}\bigl(\langle a_i,x^k\rangle - b_i - (\bar{r}^k)_i\bigr) a_i = g^k.
	\]
	This completes the proof.
\end{proof}

\subsection{Proof of Theorem \ref{main-thm-linear}}

To prove Theorem \ref{main-thm-linear}, we first need to establish several essential lemmas.

\begin{lemma}\label{lemma-in-range}
	Let $\{x^k\}_{k\geq 0}$ be the sequence generated by Algorithm \ref{RABK-VR-EFF} or Algorithm \ref{compression-RABK-EF}. Then
	$x^{k}-A^\dagger b\in\operatorname{Range}(A^\top)$  for all $k \ge 0$.
\end{lemma}
\begin{proof}
	Since $x^0 = 0 \in \operatorname{Range}(A^\top)$ and each update in Algorithms \ref{RABK-VR-EFF} and \ref{compression-RABK-EF} satisfies $g^k \in \operatorname{Range}(A^\top)$, it follows by induction that $x^k \in \operatorname{Range}(A^\top)$ for all $k$. Consequently, $x^k - A^\dagger b \in \operatorname{Range}(A^\top)$.
\end{proof}

By Proposition \ref{THM-EQ}, Algorithm \ref{RABK-VR-EFF} coincides with the iteration scheme \eqref{JacSketch}. For convenience, we adopt the notations $ {J}^k$ and $ {G}(x^k)$ from \eqref{JacSketch} in the proof of Theorem \ref{main-thm-linear}.

\begin{lemma}\label{lemma-1}
	The iteration sequence $\{x^k\}_{k\geq 0}$ generated by Algorithm \ref{RABK-VR-EFF} satisfies
	$$
	\begin{aligned}
		\mathbb{E}_{k}\left[\left\| {J}^{k+1}- {G}(x^*)\right\|_F^2\right]&=\left\|   {J}^k- {G}(x^*)  \right\|^2_{F}-\left\|   {J}^k- {G}(x^*)  \right\|^2_{P} + \left\| {G}(x^k)- {G}(x^*)\right\|^2_P,
	\end{aligned}
	$$
	where $P$ is defined as \eqref{def-D}.
\end{lemma}

\begin{proof}
	From \eqref{JacSketch}, we know that
	$$
	\begin{aligned}
		 {J}^{k+1}- {G}(x^*)&= {J}^k- {G}(x^*)+\left( {G}(x^k)- {J}^k\right) \sum_{i\in\mathcal{I}_k}e_{i} e_{i}^\top\\
		&=\left( {J}^k- {G}(x^*)\right)\left(I-\sum_{i\in\mathcal{I}_k}e_{i}e_{i}^\top\right)+\left( {G}(x^k)- {G}(x^*)\right)\sum_{i\in\mathcal{I}_k}e_{i} e_{i}^\top.
	\end{aligned}
	$$
	Since $\langle I-\sum_{i\in\mathcal{I}_k}e_{i} e_{i}^\top,\sum_{i\in\mathcal{I}_k}e_{i} e_{i}^\top\rangle=0$, we know that $\left( {J}^k- {G}(x^*)\right)\left(I-\sum_{i\in\mathcal{I}_k}e_{i} e_{i}^\top\right)$ and   $\left( {G}(x^k)- {G}(x^*)\right)\sum_{i\in\mathcal{I}_k}e_{i} e_{i}^\top$ are orthogonal. By the Pythagorean theorem and \eqref{expectation-s}, we have
	$$
	\begin{aligned}
		&\mathbb{E}_{k}\left[\left\| {J}^{k+1}- {G}(x^*)\right\|_F^2\right]\\
		=&\mathbb{E}_{k}\left[  \left\|\left( {J}^k- {G}(x^*)\right)\left(I-\sum_{i\in\mathcal{I}_k}e_{i}e_{i}^\top\right)\right\|^2_F \right]+\mathbb{E}_{k}\left[  \left\| \left( {G}(x^k)- {G}(x^*)\right)\sum_{i\in\mathcal{I}_k}e_{i}e_{i}^\top\right\|^2_F \right]\\
		=&\mathbb{E}_{k}\left[ \operatorname{Tr}\left(\left( {J}^k- {G}(x^*)\right) \left(I- \sum_{i\in\mathcal{I}_k}e_{i}e_{i}^\top\right)\left( {J}^k- {G}(x^*)\right)^\top\right) \right]\\
		&+\mathbb{E}_{k}\left[\operatorname{Tr} \left( \left( {G}(x^k)- {G}(x^*)\right)\sum_{i\in\mathcal{I}_k}e_{i}e_{i}^\top\left( {G}(x^k)- {G}(x^*)\right)^\top\right) \right]\\
		=&\left\|   {J}^k- {G}(x^*)  \right\|^2_{I-P}  + \left\| {G}(x^k)- {G}(x^*)\right\|^2_P 
	\end{aligned}
	$$
	as desired.
\end{proof}

\begin{lemma}\label{lemma-2}
	The second moment of the estimated gradient $\{g^k\}_{k\geq0}$ in Algorithm \ref{RABK-VR-EFF}  is bounded by
	$$	
	\mathbb{E}_k[\|g^k\|^2_2]\leq 2\left\| {J}^k- {G}(x^*) \right\|^2_{Q-pp^\top}+2\left\| {G}(x^k)- {G}(x^*)\right\|^2_Q,
	$$
	where $p$ and $D$ are defined as \eqref{def-D}.
\end{lemma}

\begin{proof}
	From \eqref{JacSketch}, we know that
	$$
	\begin{aligned}
		\mathbb{E}_k[\|g^k\|^2_2]&=\mathbb{E}_k\left[\left\| {J}^k p+\left( {G}(x^k)- {J}^k\right)\sum_{i\in\mathcal{I}_k}e_{i}\right\|^2_2\right]\\
		&
		=\mathbb{E}_k\left[\left\| {J}^k p+\left( {G}(x^k)- {G}(x^*)+ {G}(x^*)- {J}^k\right) \sum_{i\in\mathcal{I}_k}e_{i}\right\|^2_2\right]\\
		&\leq 2\mathbb{E}_k\left[\left\| {J}^k p-\left( {J}^k- {G}(x^*)\right) \sum_{i\in\mathcal{I}_k}e_{i}\right\|^2_2\right]+2\mathbb{E}_k\left[\left\|\left( {G}(x^k)- {G}(x^*)\right) \sum_{i\in\mathcal{I}_k}e_{i}\right\|^2_2\right]
	\end{aligned}
	$$
	Note that $ {G}(x^*)p=0$ and by \eqref{expectation-is}, we have
	$$
	\begin{aligned}
		&\mathbb{E}_k\left[\left\| {J}^k p-\left( {J}^k- {G}(x^*)\right) \sum_{i\in\mathcal{I}_k}e_{i}\right\|^2_2\right]=	\mathbb{E}_k\left[\left\|\left( {J}^k- {G}(x^*)\right) p-\left( {J}^k- {G}(x^*)\right) \sum_{i\in\mathcal{I}_k}e_{i}\right\|^2_2\right]\\
		&=\mathbb{E}_k\left[\left\|\left( {J}^k- {G}(x^*)\right) \sum_{i\in\mathcal{I}_k}e_{i}\right\|^2_2\right]-\left\|\left( {J}^k- {G}(x^*)\right) p\right\|^2_2\\
		&=   \left\|  {J}^k- {G}(x^*) \right\|^2_{Q-pp^\top},
	\end{aligned}
	$$
	which together with
	$$
	\mathbb{E}_k\left[\left\|\left( {G}(x^k)- {G}(x^*)\right) \sum_{i\in\mathcal{I}_k}e_{i}\right\|^2_2\right]=\left\| {G}(x^k)- {G}(x^*)\right\|^2_Q
	$$
	implies the lemma.
\end{proof}

Now we are ready to prove Theorem \ref{main-thm-linear}.

\begin{proof}[Proof of Theorem \ref{main-thm-linear}]
		First,  note that
		$$
		\mathbb{E}_k[g^k]= {J}^kp+\mathbb{E}_k\left[( {G}(x^k)- {J}^k)\sum_{i\in\mathcal{I}_k}e_{i}\right]= {J}^k p+( {G}(x^k)- {J}^k)p= {G}(x^k) p=\nabla f(x^k).
		$$
		We can have
		$$
		\begin{aligned}
			\mathbb{E}_k[\|x^{k+1}-x^*\|^2_2]&=\mathbb{E}_k[\|x^k-x^*-\alpha g^k\|^2_2]\\
			&=\|x^k-x^*\|^2_2-2\alpha\langle \mathbb{E}_k[g^k],x^k-x^*\rangle+\alpha^2 \mathbb{E}_k[\|g^k\|^2_2]\\
			&=\|x^k-x^*\|^2_2-2\alpha\langle \nabla f(x^k),x^k-x^*\rangle+\alpha^2 \mathbb{E}_k[\|g^k\|^2_2]\\
			&=\|x^k-x^*\|^2_2-2\alpha\langle \nabla f(x^k)-\nabla f(x^*),x^k-x^*\rangle+\alpha^2 \mathbb{E}_k[\|g^k\|^2_2],
		\end{aligned}
		$$
		where the last inequality follows from $\nabla f(x^*)=0$. Note that 
		$$
		\langle \nabla f(x^k)-\nabla f(x^*),x^k-x^*\rangle=\langle A^\top (Ax^k-Ax^*),x^k-x^*\rangle=\|A(x^k-x^*)\|^2_2,
		$$
		which together with Lemma \ref{lemma-2} implies
		$$	\begin{aligned}
			\mathbb{E}_k[\|x^{k+1}-x^*\|^2_2]&\leq \|x^k-x^*\|^2_2-2\alpha\|A(x^k-x^*)\|^2_2\\
			&
			+2\alpha^2\left(\left\| {J}^k- {G}(x^*) \right\|^2_{Q-pp^\top}+\left\| {G}(x^k)- {G}(x^*)\right\|^2_Q\right).
		\end{aligned}
		$$
		By Lemmas \ref{error-bound-lemma} and \ref{lemma-in-range}, we have 
		\begin{equation}
			\label{xie-proof-0827-1}
			\begin{aligned}
				&\mathbb{E}_k\left[\|x^{k+1}-x^*\|^2_2+\beta\alpha\left\| {J}^{k+1}- {G}(x^*) \right\|^2_F\right]\\
				\leq& (1-\alpha \sigma_{\min}^2(A))\|x^k-x^*\|^2_2-\alpha\|A(x^k-x^*)\|^2_2	+2\alpha^2 \left\| {J}^k- {G}(x^*) \right\|^2_{Q-pp^\top}
				\\
				&+ 2\alpha^2\left\| {G}(x^k)- {G}(x^*)\right\|^2_Q +\beta\alpha\left\| {J}^{k+1}- {G}(x^*)\right\|^2_F,
			\end{aligned}
		\end{equation}
		From Lemma \ref{lemma-1}, we know that
		\begin{equation}
			\label{J-ineq-11}
			\mathbb{E}_{k}\left[\left\| {J}^{k+1}- {G}(x^*)\right\|_F^2\right]\leq \left(1-\lambda_{\min}(P)\right)\left\|   {J}^k- {G}(x^*)  \right\|^2_{F}  +
			\left\| {G}(x^k)- {G}(x^*)\right\|^2_P.
		\end{equation}
		We also have
		\begin{equation}
			\label{proof-xie-lemma-826-1}
			\begin{aligned}
				\left\|\left( {G}(x^k)- {G}(x^*)\right)\right\|_P^2& 
				=\sum_{i=1}^{m}p_i \cdot\|\nabla f_{i}(x^k)-\nabla f_{i}(x^*)\|_2^2
				=\sum_{i=1}^{m}p_i \frac{\|\langle a_i,x^k-x^*\rangle a_i\|^2_2}{p_i^2}
				\\
				&= \| A(x^k-x^*)\|^2_{D}.
			\end{aligned}
		\end{equation}
		Combine \eqref{xie-proof-0827-1}, \eqref{J-ineq-11}, \eqref{proof-xie-lemma-826-1}, and note that $\left\| {J}^{k}- {G}(x^*)\right\|_F^2 =\left\| \bar{r}^k-r^* \right\|_{DP^{-1}}^2 $, we have
		$$
		\begin{aligned}
			&\mathbb{E}_k[\Phi^{k+1}] =\mathbb{E}_k\left[\|x^{k+1}-x^*\|^2_2+\beta\alpha\left\| {J}^{k+1}- {G}(x^*) \right\|^2_F\right]\\
			\leq& (1-\alpha \sigma_{\min}^2(A))\|x^k-x^*\|^2_2-\alpha\|A(x^k-x^*)\|^2_2+2\alpha^2 \left\| {J}^k- {G}(x^*) \right\|^2_{Q-pp^\top}\\
			&+ 2\alpha^2\left\| {G}(x^k)- {G}(x^*)\right\|^2_Q +\beta\alpha\left(1-\lambda_{\min}(P)\right)\left\|   {J}^k- {G}(x^*)  \right\|^2_{F}  
			+ \beta\alpha \left\| {G}(x^k)- {G}(x^*)\right\|^2_P
			\\
			\leq &(1-\alpha \sigma_{\min}^2(A))\|x^k-x^*\|^2_2
			-\left( \alpha-(2\alpha^2 \lambda_{\max}(QP^{-1})+\beta\alpha) \lambda_{\max}(D)\right)\|A(x^k-x^*)\|^2_2
			\\&+ \left(2\alpha^2\lambda_{\max}(Q-pp^\top) +\beta\alpha\left(1-\lambda_{\min}(P)\right)  \right) \left\|   {J}^k- {G}(x^*)  \right\|^2_{F}.
		\end{aligned}
		$$
		Since
		$\alpha\leq \min\left\{\frac{1-\beta\lambda_{\max}(D)}{2\lambda_{\max}(D)\lambda_{\max}(QP^{-1})}, \frac{\beta\lambda_{\min}(P)}{2\lambda_{\max}(Q-pp^\top)+\beta\sigma_{\min}^2(A)}  \right\},
		$ we know that
		$\alpha-(2\alpha^2 \lambda_{\max}(QP^{-1})+\beta\alpha) \lambda_{\max}(D)\geq0
		$ 
		and 
		$ 2\alpha^2\lambda_{\max}(Q-pp^\top) +\beta\alpha\left(1-\lambda_{\min}(P)\right) \leq  \beta\alpha(1-\alpha \sigma_{\min}^2(A)).
		$
		Hence, we have
		$
		\mathbb{E}_k[\Phi^{k+1}]\leq (1-\alpha \sigma_{\min}^2(A))\Phi^{k}.
		$
		Taking the expectation over the entire history we have
		$$
		\mathbb{E}[\Phi^{k+1}]\leq (1-\alpha \sigma_{\min}^2(A))\mathbb{E}[\Phi^{k}].
		$$
		By induction on the iteration index $k$, we can obtain the desired result.
	\end{proof}


\subsection{Proof of Theorem \ref{thm-acc}}

To prove Theorem \ref{thm-acc}, we first need to establish several essential lemmas.

\begin{lemma}\label{lemma-acc-3}
	Let $x\in\mathbb{R}^n$ be any fixed vector. Then the iteration sequence generated by Algorithm \ref{compression-RABK-EF} satisfies
	\begin{equation}
		\label{lemma-acc-3-1}
		\begin{aligned}
			\mathbb{E}_k[ \|Ay^{k+1}-b\|^2_2]
			\leq & \|Ax-b\|^2_2- \|A(x^k-x)\|^2_2+\mathbb{E}_k\left[\frac{2}{\alpha}\langle x^k-x,x^k-y^{k+1}\rangle\right.\\
			&+\left.\alpha\|A^\top (Ax^k-b)-g^k\|^2_2 + \| A(y^{k+1}-x^k)\|^2_2-\frac{1}{\alpha}\|y^{k+1}-x^k\|^2_2\right].
		\end{aligned}
	\end{equation}
\end{lemma}
\begin{proof} 
	For any $x$ and $y$, it holds that
	\begin{equation}\label{obj-function-equality}
		\|Ay-b\|^2_2- \|Ax-b\|^2_2=2\langle y-x,A^\top (Ax-b)\rangle+ \| A(y-x)\|^2_2.
	\end{equation}
	Hence, we can get 
	$$
	\begin{aligned}
		\|Ay^{k+1}-b\|^2_2=&\|Ax^k-b\|^2_2+2\langle y^{k+1}-x^k,A^\top (Ax^k-b)\rangle+ \| A(y^{k+1}-x^k)\|^2_2\\
		=&\|Ax-b\|^2_2- \|A(x^k-x)\|^2_2+2\langle x^k-x,A^\top (Ax^k-b)\rangle\\
		&+ 2\langle y^{k+1}-x^k,A^\top (Ax^k-b) \rangle  + \| A(y^{k+1}-x^k)\|^2_2\\
		=&\|Ax-b\|^2_2- \|A(x^k-x)\|^2_2+2\langle x^k-x,A^\top (Ax^k-b)-g^k+g^k\rangle\\
		&+ 2\langle y^{k+1}-x^k,A^\top (Ax^k-b) -g^k+g^k\rangle  + \| A(y^{k+1}-x^k)\|^2_2.
	\end{aligned}
	$$
	Note that  $y^{k+1}=x^k-\alpha g^k$, we have
	$$
	\begin{aligned}
		\mathbb{E}_k[\|Ay^{k+1}-b\|^2_2]
		=&\|Ax-b\|^2_2-\|A(x^k-x)\|^2_2+\mathbb{E}_k\left[\frac{2}{\alpha}\langle x^k-x,x^k-y^{k+1}\rangle-\frac{2}{\alpha}\|y^{k+1}-x^k\|^2_2\right.\\
		&+\left. 2\langle y^{k+1}-x^k+x^k-x,A^\top (Ax^k-b)-g^k\rangle  +\| A(y^{k+1}-x^k)\|^2_2\right]\\
		\leq&\|Ax-b\|^2_2- \|A(x^k-x)\|^2_2+\mathbb{E}_k\left[\frac{2}{\alpha}\langle x^k-x,x^k-y^{k+1}\rangle\right.\\
		&+\left.\alpha\|A^\top (Ax^k-b)-g^k\|^2_2 -\frac{1}{\alpha}\|y^{k+1}-x^k\|^2_2+ \| A(y^{k+1}-x^k)\|^2_2\right],
	\end{aligned}
	$$
	where the last inequality follows from 
	$$\mathbb{E}_k\left[\langle x^k-x,A^\top (Ax^k-b)-g^k\rangle \right]=\langle x^k-x,A^\top (Ax^k-b)-\mathbb{E}_k[g^k]\rangle =0$$ 
	and  the Young's inequality
	$$2\langle y^{k+1}-x^k,A^\top (Ax^k-b)-g^k\rangle\leq \alpha\|A^\top (Ax^k-b)-g^k\|^2_2+\frac{1}{\alpha}\|y^{k+1}-x^k\|^2_2.$$ 
	This completes the proof of this lemma.
\end{proof}

\begin{lemma}
	\label{lemma-acc-1}
	Suppose that the parameters satisfy $\alpha\leq\frac{1}{2\|A\|^2_2},\theta_1\leq\frac{1}{4}, \theta_2=\frac{1}{2}$, $\gamma=\frac{\alpha}{2(\theta_1+\alpha\sigma_{\min}^2(A))}$, and $\beta=1-\gamma\sigma_{\min}^2(A)$, then the iteration sequence generated by Algorithm \ref{compression-RABK-EF} satisfies
	$$
	\begin{aligned}
		\mathbb{E}_k[\mathcal{Z}^{k+1}]+\frac{ \gamma\beta }{\theta_1}\mathbb{E}_k[\mathcal{Y}^{k+1}]
		\leq&\beta\mathcal{Z}^{k}+(1-\theta_1-\theta_2)\frac{\gamma\beta}{\theta_1} \mathcal{Y}^{k}+\frac{\gamma\beta \theta_2}{\theta_1}\mathcal{Y}^{k-1} 
		-\frac{ \gamma \beta}{4\theta_1}  \|A(x^k-y^k)\|^2_2
		\\
		&-\frac{ \gamma \beta }{2\theta_1}  \|A(x^k-y^{k-1})\|^2_2+\frac{\alpha \gamma }{\theta_1}\mathbb{E}_k[\|A^\top (Ax^k-b)-g^k\|^2_2].
	\end{aligned}
	$$
\end{lemma}
\begin{proof} 
	We have
	$$
	\begin{aligned}
		\|z^{k+1}-x^*\|^2_2 
		=&  \left\|\beta z^{k}+(1-\beta)x^k+\frac{\gamma}{\alpha} (y^{k+1}-x^k)-x^*  \right\|^2_2 \\
		=&\left\| \beta(z^k-x^*)+(1-\beta)(x^k-x^*) \right\|^2_2+\frac{\gamma^2}{\alpha^2} \|y^{k+1}-x^k\|^2_2 \\
		&+  \frac{2\gamma}{\alpha}\langle y^{k+1}-x^k,\beta(z^k-x^k)+x^k-x^*\rangle \\
		=& \beta \|z^{k}-x^*\|^2_2 +(1-\beta)\|x^k-x^*\|^2_2-\beta(1-\beta)\|x^k-z^k\|^2_2\\
		&+\frac{\gamma^2}{\alpha^2}  \|y^{k+1}-x^k\|^2_2  -\frac{2\gamma}{\alpha}\langle x^k-y^{k+1},x^k-x^*\rangle +\frac{2\gamma\beta}{\alpha}\langle y^{k+1}-x^k,z^k-x^k\rangle  
		\\
		=&\beta \|z^{k}-x^*\|^2_2 +(1-\beta)\|x^k-x^*\|^2_2 -\beta(1-\beta)\|x^k-z^k\|^2_2+\frac{\gamma^2}{\alpha^2} \|y^{k+1}-x^k\|^2_2 \\
		&-\frac{2\gamma}{\alpha}\langle x^k-y^{k+1},x^k-x^*\rangle +\frac{2\gamma\beta\theta_2}{\alpha\theta_1}\langle x^k-y^{k+1},y^{k-1}-x^k\rangle  \\
		&+ \frac{2\gamma\beta(1-\theta_1-\theta_2)}{\alpha\theta_1}\langle x^k-y^{k+1},y^{k}-x^k\rangle, 
	\end{aligned}
	$$
	where the last equality follows from $x^k=\theta_1 z^k+\theta_2 y^{k-1}+(1-\theta_1-\theta_2)y^k$. Next, we will cancel the inner products in the above equation using Lemma \ref{lemma-acc-3}. By plugging $\gamma$ times \eqref{lemma-acc-3-1} (where $x=x^*$),  $\frac{\gamma\beta(1-\theta_1-\theta_2)}{\theta_1}$ times \eqref{lemma-acc-3-1} (where $x=y^{k}$), and $\frac{\gamma\beta\theta_2}{\theta_1}$ times \eqref{lemma-acc-3-1} (where $x=y^{k-1}$) into 
	the above equation, we have
	%
	%
	$$
	\begin{aligned}
		&\mathbb{E}_k[\|z^{k+1}-x^*\|^2_2]
		\\
		\leq&\beta\|z^{k}-x^*\|^2_2+(1-\beta)\|x^k-x^*\|^2_2-\gamma\|A(x^k-x^*)\|^2_2-\beta(1-\beta)\|x^k-z^k\|^2_2\\
		&+\frac{\gamma\beta(1-\theta_1-\theta_2)}{\theta_1} \mathcal{Y}^{k} 
		+\frac{\gamma\beta \theta_2}{\theta_1}\mathcal{Y}^{k-1} 
		-\frac{ \gamma\beta+ \gamma\theta_1(1-\beta)}{\theta_1}\mathbb{E}_k[\mathcal{Y}^{k+1}] 
		\\
		&+\frac{\alpha\gamma\beta+\alpha\gamma\theta_1(1-\beta)}{\theta_1} \mathbb{E}_k[\|A^\top (Ax^k-b)-g^k\|^2_2]
		-\frac{ \gamma\beta \theta_2 }{\theta_1}  \|A(x^k-y^{k-1})\|^2_2 
		\\
		&-\frac{ \gamma\beta(1-\theta_1-\theta_2)}{\theta_1} \|A(x^k-y^k)\|^2_2
		+
		\left(\frac{\gamma^2}{\alpha^2}-\frac{\gamma\beta+\gamma\theta_1(1-\beta) }{\theta_1\alpha}  \right) \mathbb{E}_k[\|y^{k+1}-x^k\|^2_2]
		\\&+\frac{\gamma\beta+\gamma\theta_1(1-\beta) }{\theta_1} \mathbb{E}_k[\|A(y^{k+1}-x^{k})\|^2_2].
	\end{aligned}
	$$
	From Lemmas \ref{error-bound-lemma} and \ref{lemma-in-range}, we have $\|A(x^k-x^*)\|^2_2\geq\sigma_{\min}^2(A)\|x^k-x^*\|^2_2$. Using $\gamma=\frac{\alpha}{2( \theta_1+\alpha\sigma_{\min}^2(A)) }$, $\beta=1-\gamma\sigma_{\min}^2(A) $,  $\alpha\leq \frac{1}{2\|A\|^2_2}$, $\theta_1\leq 1/4$, and $\theta_2=1/2$, we have
	$$
	\begin{aligned}
		\mathbb{E}_k[\mathcal{Z}^{k+1}]+\frac{ \gamma\beta }{\theta_1}\mathbb{E}_k[\mathcal{Y}^{k+1}]
		\leq&\beta\mathcal{Z}^{k}+(1-\theta_1-\theta_2)\frac{\gamma\beta}{\theta_1} \mathcal{Y}^{k}+\frac{\gamma\beta \theta_2}{\theta_1}\mathcal{Y}^{k-1} 
		-\frac{ \gamma \beta}{4\theta_1}  \|A(x^k-y^k)\|^2_2
		\\
		&-\frac{ \gamma\beta  }{2\theta_1}  \|A(x^k-y^{k-1})\|^2_2+\frac{\alpha \gamma }{\theta_1}\mathbb{E}_k[\|A^\top (Ax^k-b)-g^k\|^2_2] 
	\end{aligned}
	$$
	as desired.
\end{proof}
%

For convenience, we adopt the notations $ {J}^k$, $ {G}(y^{k-1 })$, and $ {G}(x^k)$ from \eqref{Acc-JacSketch} in the proof of Theorem  \ref{thm-acc}.
\begin{lemma}
	The iteration sequence generated by Algorithm \ref{compression-RABK-EF} satisfies
	\begin{equation}
		\label{lemma-equ-0828-2}
		\mathbb{E}_k\left[ \left\|g^k-A^\top(Ax^k-b)\right\|^2_2\right]
		\leq2\left\| {J}^k- {G}(y^{k-1 })\right\|^2_{Q-pp^\top}+2\left\| {G}(y^{k-1 })- {G}(x^{k })\right\|^2_{Q-pp^\top}
	\end{equation}
	and
	\begin{equation}
		\label{lemma-equ-0828-1}
		\begin{aligned}
			\mathbb{E}_k&\left[\left\| {J}^{k+1}- {G}(y^{k })\right\|^2_{F}\right]
			\leq \left(1-\frac{\eta\lambda_{\min} (P)}{2}\right)  \left\|\left(  {J}^{k}- {G}(y^{k-1})\right)  \right\|^2_{F}\\
			&+2\left(1+\frac{2}{\eta\lambda_{\min} (P)}\right)\left( \left\|  {G}(y^{k})- {G}(x^{k})  \right\|^2_{F}
			+ \left\|  {G}(x^{k})- {G}(y^{k-1})  \right\|^2_{F}\right).
		\end{aligned}
	\end{equation}
\end{lemma}

\begin{proof}
	From \eqref{Acc-JacSketch}, we know that $g^k= {J}^k p+( {G}(x^{k })- {J}^k) \sum_{i\in\mathcal{I}_k}e_{i}$. Since $A^\top(Ax^k-b)= {G}(x^{k }) p$, we can get
	$$
	\begin{aligned}
		\mathbb{E}_k\left[ \left\|g^k-A^\top(Ax^k-b)\right\|^2_2\right]
		=&\mathbb{E}_k\left[  \left\| {J}^k p+( {G}(x^{k })- {J}^k) \sum_{i\in\mathcal{I}_k}e_{i} - {G}(x^{k }) p\right\|^2_2\right]\\
		=& \left\| {J}^k- {G}(x^{k })\right\|^2_{Q-pp^\top}\\
		\leq&2\left\| {J}^k- {G}(y^{k-1 })\right\|^2_{Q-pp^\top}+2\left\| {G}(y^{k-1 })- {G}(x^{k })\right\|^2_{Q-pp^\top},
	\end{aligned}
	$$
	where the last inequality follows from the Cauchy-Schwarz inequality. This proves \eqref{lemma-equ-0828-2}.
	
	From Step 3 in Algorithm \ref{compression-RABK-EF}, we know that 
	$
	 {J}^{k+1}= {J}^{k}+\eta( {G}(y^{k-1})- {J}^{k}) \sum_{i\in\mathcal{I}_k}e_{i} e_{i}^\top
	$.
	Hence, we have
	$$
	\begin{aligned}
		\mathbb{E}_k\left[\left\| {J}^{k+1}- {G}(y^{k })\right\|^2_{F}\right]
		=& \mathbb{E}_k\left[\left\|\left(  {J}^{k}- {G}(y^{k-1})\right) \left(I-\eta\sum_{i\in\mathcal{I}_k}e_{i} e_{i}^\top\right)-\left( {G}(y^{k})- {G}(y^{k-1})\right) \right\|^2_{F}\right]\\
		\leq& \left(1+\frac{\eta\lambda_{\min}(P)}{2}\right)\mathbb{E}_k\left[  \left\|\left(  {J}^{k}- {G}(y^{k-1})\right) \left(I-\eta\sum_{i\in\mathcal{I}_k}e_{i} e_{i}^\top\right) \right\|^2_{F}\right]\\
		&+\left(1+\frac{2}{\eta\lambda_{\min}(P)}\right) \left\|  {G}(y^{k})- {G}(y^{k-1})  \right\|^2_{F},
	\end{aligned}
	$$
	where the last inequality due to the Young's inequality. Note that
	$$
	\begin{aligned}
		&\mathbb{E}_k\left[  \left\|\left(  {J}^{k}- {G}(y^{k-1})\right) \left(I-\eta\sum_{i\in\mathcal{I}_k}e_{i} e_{i}^\top\right) \right\|^2_{F}\right]\\
		=&\mathbb{E}_k\left[\operatorname{Tr}\left( \left(  {J}^{k}- {G}(y^{k-1})\right) \left(I-(2\eta-\eta^2)\sum_{i\in\mathcal{I}_k}e_{i} e_{i}^\top\right)   \left(  {J}^{k}- {G}(y^{k-1})\right)^\top \right)\right]\\
		=&\operatorname{Tr}\left( \left(  {J}^{k}- {G}(y^{k-1})\right) \left(I-(2\eta-\eta^2)P\right)\left(  {J}^{k}- {G}(y^{k-1})\right)^\top  \right)\\
		\leq&\left(1-(2\eta-\eta^2)\lambda_{\min} (P)\right)  \left\|\left(  {J}^{k}- {G}(y^{k-1})\right)  \right\|^2_{F},
	\end{aligned}
	$$
	which together with $\left(1+\frac{\eta\lambda_{\min}(P)}{2}\right)\left(1-(2\eta-\eta^2)\lambda_{\min} (P)\right)\leq \left(1-\frac{\eta\lambda_{\min}(P)}{2}\right)$ as $\eta<1$ and
	$$\left\|  {G}(y^{k})- {G}(y^{k-1})  \right\|^2_{F}\leq 2\left\|  {G}(y^{k})- {G}(x^{k})  \right\|^2_{F}
	+ 2\left\|  {G}(x^{k})- {G}(y^{k-1})  \right\|^2_{F}$$ implies that \eqref{lemma-equ-0828-1} holds.
\end{proof}

\begin{proof}[Proof of Theorem \ref{thm-acc}]

	%
	%
	%
	%
	By Lemma \ref{lemma-acc-1}, we know that 
	$$
	\begin{aligned}
		\mathbb{E}_k\left[\Psi^{k+1}\right]:=&\mathbb{E}_k\left[ \mathcal{Z}^{k+1} +\frac{ \gamma\beta }{\theta_1} \mathcal{Y}^{k+1}  +\frac{\gamma\beta \theta_2(1+\theta_1)}{\theta_1} \mathcal{Y}^{k}+\frac{8\alpha \gamma\lambda_{\max}(Q-pp^\top ) }{\theta_1 \eta\lambda_{\min}(P)}\left\| {J}^{k+1}- {G}(y^{k  })\right\|^2_{F}\right]\\
		\leq&\beta\mathcal{Z}^{k} +(1-\theta_1-\theta_2)\frac{\gamma\beta}{\theta_1} \mathcal{Y}^k
		+\frac{\alpha \gamma }{\theta_1} \mathbb{E}_k[\|A^\top (Ax^k-b)-g^k\|^2_2] +\frac{\gamma\beta \theta_2}{\theta_1}\mathcal{Y}^{k-1}
		\\&-\frac{ \gamma\beta  }{2\theta_1}   \|A(x^k-y^{k-1})\|^2_2  -\frac{ \gamma\beta }{4\theta_1}  \|A(x^k-y^k)\|^2_2\\
		&+ \frac{\gamma\beta \theta_2(1+\theta_1)}{\theta_1} \mathcal{Y}^{k} +\frac{8\alpha \gamma \lambda_{\max}(Q-pp^\top )}{\theta_1 \eta\lambda_{\min}(P)}\mathbb{E}_k\left[\left\| {J}^{k+1}- {G}(y^{k})\right\|^2_{F}\right] \\
		\leq&\beta\mathcal{Z}^{k} +\left(1-\frac{\theta_1}{2}\right)\frac{\gamma\beta}{\theta_1} \mathcal{Y}^k
		+\left(1-\frac{\theta_1}{2}\right)\frac{\gamma\beta \theta_2(1+\theta_1)}{\theta_1}\mathcal{Y}^{k-1}-\frac{ \gamma \beta }{2\theta_1}  \|A(x^k-y^{k-1})\|^2_2
		\\&+\frac{2\alpha \gamma }{\theta_1}\left\| {J}^k- {G}(y^{k-1 })\right\|^2_{Q-pp^\top}+ \frac{2\alpha \gamma }{\theta_1}\left\| {G}(y^{k-1 })- {G}(x^{k })\right\|^2_{Q-pp^\top} \\
		&-\frac{ \gamma\beta }{4\theta_1} \|A(x^k-y^k)\|^2_2 +\frac{8\alpha \gamma\lambda_{\max}(Q-pp^\top ) }{\theta_1 \eta\lambda_{\min}(P)}\left(1-\frac{\eta\lambda_{\min} (P)}{2}\right)\left\|\left(  {J}^{k}- {G}(y^{k-1})\right)  \right\|^2_{F}\\
		& +\frac{16\alpha \gamma\lambda_{\max}(Q-pp^\top ) }{\theta_1\eta\lambda_{\min}(P)}\left(1+\frac{2}{\eta\lambda_{\min} (P)}\right)  \left\|  {G}(y^{k})- {G}(x^{k})  \right\|^2_{F}  \\
		&+\frac{16\alpha \gamma\lambda_{\max}(Q-pp^\top ) }{\theta_1\eta\lambda_{\min}(P)}\left(1+\frac{2}{\eta\lambda_{\min} (P)}\right) \left\|  {G}(x^{k})- {G}(y^{k-1})  \right\|^2_{F},
	\end{aligned}
	$$
	where the last inequality follows from \eqref{lemma-equ-0828-2} and \eqref{lemma-equ-0828-1}. 
	Since 
	$$
	\begin{aligned}
		\left\|  {G}(y^{k})- {G}(x^{k})  \right\|^2_{F}
		 \leq \lambda_{\max}(DP^{-1})  \|A(y^k-x^k)\|^2_2.
	\end{aligned}
	$$
	Let  $\rho:=\lambda_{\max}(DP^{-1}) $.
	Similarly, we have $\left\| {G}(x^{k}) -  {G}(y^{k-1})\right\|^2_{F}\leq\rho\|A(x^k-y^{k-1})\|^2_2$. Then we can get
	$$
	\begin{aligned}
		\mathbb{E}_k\left[\Psi^{k+1}\right]
		\leq&\beta\mathcal{Z}^{k} +\left(1-\frac{\theta_1}{2}\right)\frac{\gamma\beta}{\theta_1} \mathcal{Y}^k
		+\left(1-\frac{\theta_1}{2}\right)\frac{\gamma\beta \theta_2(1+\theta_1)}{\theta_1}\mathcal{Y}^{k-1}\\
		&+\left(1-\frac{\eta\lambda_{\min} (P)}{4}\right) \frac{8\alpha \gamma\lambda_{\max}(Q-pp^\top) }{\theta_1\eta\lambda_{\min}(P)}\left\| {J}^k- {G}(y^{k-1 })\right\|^2_{F}
		\\&-\frac{ \gamma  }{\theta_1}\left( \frac{\beta}{2}-2\alpha\rho\lambda_{\max}(Q-pp^\top) \left( 1+\frac{8}{\eta\lambda_{\min}(P)} + \frac{16}{\eta^2\lambda^2_{\min} (P)}\right)\right) \|A(x^k-y^{k-1})\|^2_2   \\
		&-\frac{ \gamma }{\theta_1} \left(\frac{\beta}{4} -4\alpha \rho\lambda_{\max}(Q-pp^\top)  \left(\frac{4}{\eta\lambda_{\min}(P)}+\frac{8}{\eta^2\lambda^2_{\min} (P)}\right) \right)\|A(x^k-y^k)\|^2_2.
	\end{aligned}
	$$
	Since $\eta =\frac{ 1}{\lambda_{\min}(P)+1}$, $\beta
	=1-\frac{\alpha\sigma^2_{\min}(A)}{2(\theta_1+\alpha\sigma^2_{\min}(A))}
	\geq\frac12$, and  
	$$
	\alpha=\min\left\{\frac{(1-\eta)^{2}}{128\rho(3-\eta)\lambda_{\max}(Q-pp^\top) },\frac{1}{4\|A\|^2_2}\right\},
	$$
	we know that
	$$
	\frac{\beta}{2}-2\alpha \rho\lambda_{\max}(Q-pp^\top) \left( 1+\frac{8}{\eta\lambda_{\min}(P)} + \frac{16}{\eta^2\lambda^2_{\min} (P)}\right)\geq0
	$$
	and
	$$
	\frac{\beta}{4} -4\alpha\rho\lambda_{\max}(Q-pp^\top) \left(\frac{4}{\eta\lambda_{\min}(P)}+\frac{8}{\eta^2\lambda^2_{\min} (P)}\right) \geq0.
	$$
	Note that $\gamma=\frac{\alpha}{2(\theta_1+\alpha\sigma_{\min}^2(A))}$, $\beta=1-\gamma\sigma_{\min}^2(A)\leq 1-\frac{\alpha\sigma_{\min}^2(A)}{4\theta_1}$ due to $0<\alpha\sigma_{\min}^2(A)\leq \frac{1}{4}$ and $\alpha\sigma_{\min}^2(A)\leq \theta_1$. Thus we have
	$$
	\begin{aligned}
		\mathbb{E}_k\left[\Psi^{k+1}\right]\leq &\beta\mathcal{Z}^{k} +\left(1-\frac{\theta_1}{2}\right)\frac{\gamma\beta}{\theta_1} \mathcal{Y}^k
		+\left(1-\frac{\theta_1}{2}\right)\frac{\gamma\beta \theta_2(1+\theta_1)}{\theta_1}\mathcal{Y}^{k-1}\\
		&+\left(1-\frac{\eta\lambda_{\min} (P)}{4}\right) \frac{8\alpha \gamma\lambda_{\max}(Q-pp^\top) }{\theta_1 \eta\lambda_{\min}(P)}\left\| {J}^k- {G}(y^{k-1 })\right\|^2_{F}
		\\
		\leq &\left(1-\frac{\alpha\sigma_{\min}^2(A)}{4\theta_1}\right)\mathcal{Z}^{k} +\left(1-\frac{\theta_1}{2}\right)\frac{\gamma\beta}{\theta_1} \mathcal{Y}^k
		+\left(1-\frac{\theta_1}{2}\right)\frac{\gamma\beta \theta_2(1+\theta_1)}{\theta_1}\mathcal{Y}^{k-1}\\
		&+\left(1-\frac{\eta\lambda_{\min} (P)}{4}\right) \frac{8\alpha \gamma\lambda_{\max}(Q-pp^\top) }{\theta_1 \eta\lambda_{\min}(P)}\left\| {J}^k- {G}(y^{k-1 })\right\|^2_{F}\\
		\leq& \left(1-\min\left\{ \frac{\eta\lambda_{\min} (P)}{4},\frac{\sqrt{\alpha\sigma_{\min}^2(A)}}{4}  \right\}\right)\Psi^k,
	\end{aligned}
	$$
	where the last inequality follows from that
	$
	\theta_1=\min\left\{\frac{1}{4},\sqrt{\frac{\alpha\sigma_{\min}^2(A)}{2}} \right\}
	$.
	%
	%
\end{proof}

\end{document}